\documentclass[11pt]{article}

\usepackage[english]{babel}
\usepackage[T1]{fontenc}
\usepackage[utf8]{inputenc}
\usepackage{lmodern}
\usepackage{amsmath,amssymb,amsthm,mathtools,mathrsfs}
\usepackage{aliascnt}
\usepackage{esint}
\usepackage{geometry}
\usepackage{microtype}
\usepackage{enumitem}
\usepackage{booktabs}
\usepackage{csquotes}
\usepackage[hidelinks]{hyperref}
\usepackage[nameinlink,noabbrev]{cleveref}
\usepackage[backend=biber,style=numeric,sorting=nyt]{biblatex}
\usepackage{titlesec}

\titleformat{\part}[block]
  {\normalfont\Large\bfseries}
  {}
  {0pt}
  {}

\newtheorem{theorem}{Theorem}[section]
\newaliascnt{proposition}{theorem}
\newtheorem{proposition}[proposition]{Proposition}
\aliascntresetthe{proposition}
\newaliascnt{lemma}{theorem}
\newtheorem{lemma}[lemma]{Lemma}
\aliascntresetthe{lemma}
\newaliascnt{corollary}{theorem}
\newtheorem{corollary}[corollary]{Corollary}
\aliascntresetthe{corollary}
\theoremstyle{definition}
\newaliascnt{definition}{theorem}
\newtheorem{definition}[definition]{Definition}
\aliascntresetthe{definition}
\newaliascnt{example}{theorem}

\aliascntresetthe{example}
\theoremstyle{remark}
\newaliascnt{remark}{theorem}
\newtheorem{remark}[remark]{Remark}
\aliascntresetthe{remark}

\crefname{theorem}{Theorem}{Theorems}
\Crefname{theorem}{Theorem}{Theorems}
\crefname{proposition}{Proposition}{Propositions}
\Crefname{proposition}{Proposition}{Propositions}
\crefname{lemma}{Lemma}{Lemmas}
\Crefname{lemma}{Lemma}{Lemmas}
\crefname{corollary}{Corollary}{Corollaries}
\Crefname{corollary}{Corollary}{Corollaries}
\crefname{definition}{Definition}{Definitions}
\Crefname{definition}{Definition}{Definitions}
\crefname{example}{Example}{Examples}
\Crefname{example}{Example}{Examples}
\crefname{remark}{Remark}{Remarks}
\Crefname{remark}{Remark}{Remarks}
\crefname{section}{Section}{Sections}
\Crefname{section}{Section}{Sections}

\title{Normal-Direction Energy and Fourier Restriction for Convex Planar Curves}
\author{Vicente Vergara\footnote{Department of Mathematics, Faculty of Physical and Mathematical Sciences, University of Concepción, Concepción, Chile. \texttt{vvergaraa@udec.cl}}}
\date{}

\begin{document}
\maketitle

\begin{abstract}
Let $\Gamma\subset\mathbb R^2$ be a compact convex $C^2$ curve with no affine subarcs, with arc-length measure $\mathrm d\sigma$, scalar curvature $\kappa$, and Gauss map $\mathcal N$, and let $E_\Gamma$ be its extension operator. If $\mathrm d\mu_\kappa=\kappa\,\mathrm d\sigma$ is the turning measure and $\mu_\kappa$ is doubling with constant $D$, let $B_\Gamma$ denote intrinsic arc-length balls and, for $R\ge1$, define a tangential scale $r_R(\xi)$ by
\[
r_R(\xi)\,\mu_\kappa(B_\Gamma(\xi,r_R(\xi)))\asymp R^{-1},
\qquad
\rho_R(\xi)=\frac{R^{-1}}{r_R(\xi)}.
\]
Up to structural constants, $r_R$ is the largest scale on which the curve can be linearized to precision $R^{-1}$. From $\rho_R$ we introduce an energy $\mathcal E_R^\Gamma(f)$ of the pushforward measure $\nu_f=\mathcal N_\#(|f|^2\,\mathrm d\sigma)$ on normal directions and prove, for every $x_0\in\mathbb R^2$,
\[
\int_{B(x_0,R)}|E_\Gamma f(x)|^4\,\mathrm dx
\lesssim_{\Gamma,D}\mathcal E_R^\Gamma(f).
\]
For uniformly nondegenerate curvature, this inequality recovers Tomas's local logarithmic loss for the circle and gives the same bound for the $C^2$ class considered here. If the normal directions of $\operatorname{supp}f$ occupy an angular arc of length $\delta$, the factor $1+\log R$ is replaced by $1+\log(1+\delta R^{1/2})$.

For the monomial curves $\gamma_k(t)=(t,t^k)$, $k\ge3$ real, the scales are explicit and the energy has a multiscale representation in terms of masses of pairs of normal directions. Consequently, a sufficiently large lower bound for the local $L^4$ norm of $E_{\gamma_k}f$ forces quantitative concentration of $\nu_f$ at some angular scale. Finally, under an $s$-dimensional Frostman condition, $0<s\le1$, we obtain explicit growth estimates with critical transition $s_c(k)=\frac{k-2}{3k-4}$.
The powers of $R$ and the critical logarithmic factor are sharp at the energy level; no corresponding optimality is claimed for the restriction inequality.
\end{abstract}



\section{Introduction}\label{sec:intro}

Let $\Gamma\subset\mathbb R^2$ be a compact regular curve with arc-length measure $\mathrm d\sigma$, and let
\[
E_\Gamma f(x)=\int_\Gamma e^{2\pi i x\cdot\xi}f(\xi)\,\mathrm d\sigma(\xi)
\]
be the associated extension operator. In local $L^4$ problems, the total mass $\|f\|_2^2$ does not distinguish data with the same mass but very different directional distributions. We retain this directional information through the Gauss map $\mathcal N:\Gamma\to\mathbb S^1$, given by the unit normal, and the associated pushforward measure on normal directions
\[
\nu_f=\mathcal N_\#\bigl(|f|^2\,\mathrm d\sigma\bigr).
\]
The local $L^4$ estimates are organized by a Riesz-type energy of this measure, with the transverse singularity regularized at a terminal angular scale determined by the local geometry of the curve.

For the circle, Tomas proved the local endpoint estimate~\cite{Tomas1980}:
\[
\|E_{\mathbb S^1}f\|_{L^4(B(0,R))}
\lesssim
(1+\log R)^{1/4}\|f\|_4.
\]
A consequence of the general theorem proved here gives the same logarithmic loss for the uniformly nondegenerate $C^2$ curves in our class. The energy also retains the distribution of the normal directions and allows the global logarithm to be replaced by the angular range effectively occupied by the data.

The angular regularization comes from a scale determined by the turning measure. The monomial model makes this scale explicit and allows the energy to be analyzed across angular scales.

\subsection*{Terminal scale and general $L^4$ estimates}

Let $\Gamma$ be a compact convex $C^2$ curve with no affine subarcs. Denote by $T$ the unit tangent and by $\kappa\ge0$ the scalar curvature. In a local arc-length parametrization,
\[
\kappa(\gamma(t))=|T'(t)|=
\left|\frac{\mathrm d}{\mathrm dt}\mathcal N(\gamma(t))\right|.
\]
The geometric measure controlling the terminal scale is the turning measure $\mathrm d\mu_\kappa=\kappa\,\mathrm d\sigma$.
Unlike $\nu_f$, the measure $\mu_\kappa$ depends only on the curve. If $B_\Gamma(\xi,r)$ is the intrinsic arc-length ball and $\mu_\kappa$ is doubling with constant $D$, define the terminal tangential scale through
\begin{equation}\label{eq:intro-terminal-general}
r_R(\xi)\,\mu_\kappa(B_\Gamma(\xi,r_R(\xi)))\asymp_\Gamma R^{-1},
\end{equation}
and set $\rho_R(\xi)=\frac{R^{-1}}{r_R(\xi)}$.
For a terminal cap $I$ one obtains
\[
|I|\operatorname{diam}\mathcal N(I)\asymp_{\Gamma,D}R^{-1}.
\]
The scale $r_R$ also has a geometric characterization. If
\[
\mathfrak e_\Gamma(\xi,r)
=
\sup_{\eta\in B_\Gamma(\xi,r)}
\operatorname{dist}\bigl(\eta,\xi+\mathbb RT(\xi)\bigr),
\]
then, under doubling,
\[
\mathfrak e_\Gamma(\xi,r)
\asymp_{\Gamma,D}
r\,\mu_\kappa(B_\Gamma(\xi,r)).
\]
Consequently, Proposition~\ref{prop:maximal-tangential-scale} identifies $r_R(\xi)$, up to structural constants, with the largest tangential scale on which the curve can be linearized to precision $R^{-1}$. In nondegenerate regions one recovers $r_R\asymp R^{-1/2}$; near a flat point with $\kappa(t_0+u)\asymp u^m$, one obtains $r_R(t_0)\asymp R^{-1/(m+2)}$.

The doubling hypothesis has three additional roles. First, it yields slow variation of $r_R$ and $\rho_R$, allowing the construction of a terminal partition with comparable consecutive caps. Second, on a graph chart $\gamma(t)=(t,\phi(t))$, for $u<v$ set $\bar t=(u+v)/2$ and $\tilde t=(v-u)/2$, and define
\[
q_\phi(\bar t,\tilde t)
=
\phi(\bar t+\tilde t)+\phi(\bar t-\tilde t)-2\phi(\bar t).
\]
Then
\[
q_\phi(\bar t,\tilde t)
\asymp_{\Gamma,D}
\tilde t\,\partial_{\tilde t}q_\phi(\bar t,\tilde t),
\]
a relation that quantitatively separates the near-diagonal and transverse regimes. Third, it provides the lower growth needed to compare terminal scales at different radii. The transverse geometry and terminal partition lead to uniformly bounded multiplicity of the frequency sumsets and, by Plancherel and Cauchy--Schwarz, to the corresponding essential biorthogonality.

The general energy associated with this geometry is
\[
\mathcal E_R^\Gamma(f)
=
\iint_{\Gamma^2}
\frac{|f(\xi)|^2|f(\eta)|^2}
{|\mathcal N(\xi)-\mathcal N(\eta)|+
\max\{\rho_R(\xi),\rho_R(\eta)\}}
\,\mathrm d\sigma(\xi)\mathrm d\sigma(\eta).
\]
Within the family of Riesz-type kernels with cutoff constant on each cap, the canonical angular cutoff has the largest admissible order: any competing cutoff $\widehat\rho_I$ that uniformly dominates the diagonal cost must satisfy
\[
\widehat\rho_I\lesssim_{\Gamma,D}(R|I|)^{-1}
\asymp_{\Gamma,D}\operatorname{diam}\mathcal N(I).
\]
Thus the canonical cutoff has the largest order compatible with that cost. This does not imply uniqueness of the kernel: decreasing the cutoff only increases the energy. A structural characterization that essentially singles out this kernel within broader classes of regularizations lies beyond the scope of this paper.

The main $L^4$ result is the following.

\begin{theorem}[General $L^4$ estimate and weighted threshold]\label{thm:intro-direct-general}
Let $\Gamma$ be a compact convex $C^2$ curve with no affine subarcs and whose turning measure is doubling with constant $D$. Then, for every $R\ge1$ and $x_0\in\mathbb R^2$,
\[
\int_{B(x_0,R)}|E_\Gamma f(x)|^4\,\mathrm dx
\lesssim_{\Gamma,D}
\mathcal E_R^\Gamma(f).
\]
If $Q_D=\log_2D$ and $\beta_D=\frac{Q_D}{Q_D+1}$, then, for every $\iota>\beta_D$,
\[
\int_{\mathbb R^2}|E_\Gamma f(x)|^4
\left(1+\frac{|x-x_0|}{R}\right)^{-\iota}\,\mathrm dx
\lesssim_{\Gamma,D,\iota}
\mathcal E_R^\Gamma(f).
\]
If in addition $\Gamma$ is equipped with quantitative finite-type data whose uniform order bound is $m_*$, the same estimate holds for $\iota>\beta_*:=\frac{m_*+1}{m_*+2}$.
For a fixed finite-type curve, if $m_{\max}(\Gamma)$ is the maximal order actually attained and $\beta_{\max}(\Gamma)=\frac{m_{\max}(\Gamma)+1}{m_{\max}(\Gamma)+2}$, the inequality holds for $\iota>\beta_{\max}(\Gamma)$ and fails for $\iota<\beta_{\max}(\Gamma)$. Thus, for a fixed curve, only the endpoint $\iota=\beta_{\max}(\Gamma)$ remains open.
\end{theorem}

The ball estimate is Theorem~\ref{thm:direct-general-ball}; the weighted version under doubling is Corollary~\ref{cor:weighted-doubling}; the quantitative improvement from finite-type data is Corollary~\ref{cor:weighted-general}; and the subendpoint obstruction is Proposition~\ref{prop:weighted-threshold-necessary}.

In the uniformly nondegenerate regime, $r_R\asymp R^{-1/2}$ and $\rho_R\asymp R^{-1/2}$. Corollary~\ref{cor:nondegenerate-logarithmic-bound} then compares the energy with a regularized angular convolution and gives
\[
\mathcal E_R^\Gamma(f)
\lesssim_\Gamma
(1+\log R)\|f\|_4^4.
\]
Combining this with the ball estimate gives
\[
\|E_\Gamma f\|_{L^4(B(x_0,R))}
\lesssim_\Gamma
(1+\log R)^{1/4}\|f\|_4.
\]
For the circle, this recovers the estimate displayed above. More generally, the same logarithmic loss follows, as a consequence of the general theorem, for every curve in the preceding class whose curvature is bounded above and below by positive constants. The same argument retains information discarded by the final norm bound: if the normal directions of $\operatorname{supp}f$ are contained in an angular arc of length $\delta\le1$, the energy factor is reduced to $1+\log(1+\delta R^{1/2})$, and for $\delta\ge R^{-1/2}$ the logarithmic loss depends, up to constants, only on the angular range effectively occupied.

The $L^4$ estimate above uses only $C^2$ regularity. The class $C^\infty$ enters in an additional observation: for a fixed smooth convex curve with no affine subarcs, doubling rules out curvature zeros of infinite order and forces finite type; in that case the data may be chosen with $m_*=m_{\max}(\Gamma)$. Remark~\ref{rem:infinite-flat-nondoubling} shows that this exclusion is genuine: the arc-length construction with curvature $\kappa(t)=e^{-1/t^2}$ for $t\neq0$ and $\kappa(0)=0$ gives a $C^\infty$ convex curve with no affine subarcs whose turning measure is not doubling. For families, by contrast, a uniform bound $m_*$ is additional quantitative information and determines a necessary threshold for the class only when that order is attained by some member. In that case the uniform endpoint $\iota=\beta_*$ also remains unresolved.

\subsection*{The monomial model as an explicit case}

Consider the monomial family
\[
\gamma_k(t)=(t,t^k),
\qquad 0\le t\le1,
\qquad k\ge3\ \text{real}.
\]
Denote its Gauss map by $\mathcal N_k$. This family belongs to the $C^2$ class required by the general $L^4$ estimate even when $k$ is not an integer. Its turning measure satisfies $\mathrm d\mu_\kappa\asymp_k t^{k-2}\,\mathrm dt$ and is doubling. The equation defining the terminal scale yields, uniformly in $t\in[0,1]$,
\[
r_R(\gamma_k(t))\asymp_k\ell_R(t),
\qquad
\rho_R(\gamma_k(t))\asymp_k\tau_R(t),
\]
where
\[
\ell_R(t)=R^{-1/2}(t+R^{-1/k})^{-(k-2)/2},
\qquad
\tau_R(t)=R^{-1/2}(t+R^{-1/k})^{(k-2)/2},
\]
and $\ell_R(t)\tau_R(t)=R^{-1}$. Thus the general terminal scale recovers the canonical monomial covering: near the flat point $r_R\asymp R^{-1/k}$ and $\rho_R\asymp R^{-(k-1)/k}$, while away from that region one recovers the parabolic behavior weighted by position.

Define
\begin{equation}\label{eq:intro-energy-monomial}
\mathcal E_{R,k}(f)
=
\iint_{[0,1]^2}
\frac{|f(u)|^2|f(v)|^2}
{|\mathcal N_k(u)-\mathcal N_k(v)|+\tau_R((u+v)/2)}
\,\mathrm d\sigma(u)\mathrm d\sigma(v).
\end{equation}
Comparison between the midpoint cutoff and the cutoffs associated with the endpoints identifies this energy, up to constants depending on $k$, with the specialization of $\mathcal E_R^\Gamma$ to $\gamma_k$. Thus the estimate on balls is a direct consequence of Theorem~\ref{thm:direct-general-ball}; the explicit formula for $\tau_R$ also allows radii to be compared with the exact exponent $(k-1)/k$.

\begin{theorem}[Monomial $L^4$ estimate]\label{thm:intro-direct-monomial}
Let $k\ge3$ be real, $R\ge1$, and $\iota>\frac{k-1}{k}$. Then, for every $x_0\in\mathbb R^2$,
\[
\int_{\mathbb R^2}|E_{\gamma_k}f(x)|^4
\left(1+\frac{|x-x_0|}{R}\right)^{-\iota}\,\mathrm dx
\lesssim_{k,\iota}
\mathcal E_{R,k}(f).
\]
\end{theorem}

The ball estimate is Theorem~\ref{thm:ball-direct-monomial}, a specialization of the general theory, while Corollary~\ref{cor:weighted-direct-monomial} uses the explicit radius comparison from Proposition~\ref{prop:radius-comparison-monomial}. Proposition~\ref{prop:weighted-threshold-necessary}, applied with $m=k-2$, shows that the uniform inequality fails for $\iota<(k-1)/k$. The endpoint $\iota=(k-1)/k$ remains open.

\subsection*{Multiscale structure and the Frostman regime}

The monomial model allows a finer analysis of the energy beyond the $L^4$ estimate because the terminal cutoff can be described explicitly on the normal arc. Since $\mathcal N_k$ is injective on $[0,1]$, set
\[
\Omega_k=\mathcal N_k([0,1]),
\qquad
\rho_R(\mathcal N_k(t))=\tau_R(t).
\]
For a finite measure $\nu$ on $\Omega_k$, define
\begin{equation}\label{eq:intro-intrinsic-energy}
\mathcal I_R(\nu)
=
\iint_{\Omega_k^2}
\frac{\mathrm d\nu(\omega)\mathrm d\nu(\omega')}
{|\omega-\omega'|+\max\{\rho_R(\omega),\rho_R(\omega')\}}.
\end{equation}
For $\nu=\nu_f$,
\[
\mathcal E_{R,k}(f)\asymp_k\mathcal I_R(\nu_f).
\]

If
\[
\Delta_R(\omega,\omega')
=
|\omega-\omega'|+\max\{\rho_R(\omega),\rho_R(\omega')\}
\]
and
\[
\widetilde P_R(r;\nu)
=
(\nu\times\nu)\{(\omega,\omega'): \Delta_R(\omega,\omega')\le r\},
\]
the layer-cake identity gives exactly
\[
\mathcal I_R(\nu)
=
\int_0^\infty
\widetilde P_R(r;\nu)\,\frac{\mathrm dr}{r^2}.
\]
After dyadic discretization, the normalized energy is comparable to a sum of normalized pair masses. This allows one to select a scale at which the pushforward measure has quantitative concentration. In particular, Corollary~\ref{cor:extension-concentration-monomial} combines this selection with Theorem~\ref{thm:ball-direct-monomial}: there exist $\Lambda_0(k)$ and $C(k)$ such that, if $\Lambda\ge\Lambda_0(k)$ and
\[
\int_{B(x_0,R)}|E_{\gamma_k}f(x)|^4\,\mathrm dx
\ge\Lambda\|f\|_2^4,
\]
then for some relevant scale $r$ and some $\omega_0\in\Omega_k$,
\[
\nu_f(B(\omega_0,Cr))
\gtrsim_k
\frac{\Lambda}{1+\log R}\,r\|f\|_2^2.
\]
The universal loss $1+\log R$ may be replaced by the effective number of scales of $\nu_f$. This conclusion is not a classification of near-extremizers and is stated only for the monomial model.

For upper bounds, a Frostman condition on $\nu$ controls the same energy and hence, by the $L^4$ estimate, the extension operator. Suppose that, for some $0<s\le1$ and $F_s>0$,
\[
\nu(B(\omega,r))
\le
F_s\,\nu(\Omega_k)\,r^s,
\qquad 0<r\le1.
\]
The geometry of the sublevel sets of $\rho_R$ produces the threshold $s_c(k)=\frac{k-2}{3k-4}$. Theorem~\ref{thm:frostman-phase} states, for $0<s<1$,
\[
\frac{\mathcal I_R(\nu)}{\nu(\Omega_k)^2}
\lesssim_{k,s}
F_s^2
\begin{cases}
R^{\frac{k-1}{k}(1-2s)},&0<s<s_c(k),\\[1mm]
R^{\frac{k-1}{3k-4}}(1+\log R),&s=s_c(k),\\[1mm]
R^{(1-s)/2},&s_c(k)<s<1,
\end{cases}
\]
while for $s=1$,
\[
\mathcal I_R(\nu)
\lesssim_k
F_1^2\nu(\Omega_k)^2(1+\log R).
\]
The powers of $R$ and the logarithmic growth in the critical case are sharp at the energy level on Ahlfors-regular angular models, while no corresponding sharpness is claimed for the resulting restriction inequality. Corollary~\ref{cor:frostman-restriction} combines these bounds with the monomial $L^4$ estimate to obtain the corresponding restriction estimates. The optimal dependence on $F_s$ also remains open; the intermediate pair-mass estimate retains finer information than the final uniform bound, which carries the factor $F_s^2$. For the classical framework of Frostman measures, Riesz energies, and capacities, see Mattila \cite{Mattila1995}; for recent results on Frostman measures on curves of nonvanishing curvature, see \cite{DasuDemeter2024,OrponenPuliattiPyorala2025,Yi2025}.

Once $\rho_R$ is fixed, the layer-cake identity, dyadic discretization, and selection of one scale do not use any further monomial structure. Extending the Frostman analysis to a general finite-type curve requires a sufficiently precise intrinsic description of the sublevel sets
\[
\Omega_R(r)=\{\omega:\rho_R(\omega)\le r\}
\]
that reproduces the pair-mass calculation and the Frostman energy bounds of the monomial model. The present theory does not identify an intrinsic analogue of the monomial quantity $L_R(r):=\operatorname{diam}\Omega_R(r)$ in that setting and therefore does not produce corresponding general Frostman energy bounds. This is the main boundary in the paper between the general $L^4$ estimate and the multiscale and Frostman analysis developed here.

The framework also uses convexity and injectivity of the Gauss map in an essential way. For nonconvex curves or curves with a noninjective Gauss map, the measure $\mathcal N_\#(|f|^2\,\mathrm d\sigma)$ loses information about the branch of origin, so an extension would require additional geometric data beyond the angular pushforward. Curves with sign-changing curvature likewise fall outside the present scope.

\subsection*{Related work}

Fourier restriction and oscillatory-integral theory belong to the Stein--Tomas tradition~\cite{Stein1993,Tomas1975}. In dimension two, Zygmund is a foundational precursor for restriction to the circle~\cite{Zygmund1974}, and Sjölin obtained results for convex planar curves~\cite{Sjolin1974}. Restriction for degenerate and finite-type curves was developed, among others, by Christ and Drury--Marshall~\cite{Christ1985,DruryMarshall1987}; for uniform local estimates with affine arc-length measure in finite-type classes, see Dendrinos--Müller~\cite{DendrinosMuller2013}.

The argument also uses square-function and orthogonality estimates. The Córdoba--Fefferman mechanism~\cite{CordobaFefferman1978} and the bilinear approach of Tao--Vargas--Vega~\cite{TaoVargasVega1998} organize interactions between separated frequency pieces. Biggs--Brandes--Hughes prove a local $L^4$ square-function estimate for finite-type planar curves~\cite{BiggsBrandesHughes2026}. For monomial models with integer exponent, Schippa constructs a covering adapted to the largest linearization scale and proves a sharp square-function estimate together with the corresponding essential biorthogonality~\cite{Schippa2024}. Bulj--Inami--Shiraki obtain $L^4$ reverse square functions for power curves and quantify the loss associated with the decomposition scale~\cite{BuljInamiShiraki2026}. For decoupling of nondegenerate curves, see Bourgain--Demeter~\cite{BourgainDemeter2017}; Guth's polynomial partitioning belongs to the broader framework of geometric extension arguments in higher dimensions~\cite{Guth2016}.

The Gauss map and normal geometry also enter extension identities and estimates in Bennett--Nakamura--Shiraki~\cite{BennettNakamuraShiraki2024}, Bennett--Gutiérrez--Nakamura--Oliveira~\cite{BennettGutierrezNakamuraOliveira2025}, and Bulj--Shiraki~\cite{BuljShiraki2026}. Here the measure $\nu_f$ enters through an energy whose kernel retains the transverse singularity and regularizes it with the angular cutoff determined by the terminal scale.

\subsection*{Organization}

The paper is organized as follows. Section~\ref{sec:terminal-geometry} develops the geometry of the terminal scale: bilinear preliminaries, turning measure, linearization, finite type, partition, and transverse geometry. Section~\ref{sec:direct-general} proves bounded multiplicity of the sumsets, essential biorthogonality, the energy--sum equivalence, the general $L^4$ estimate, the maximal admissible size of the cutoff in the local kernel family, and radius comparisons. Section~\ref{sec:monomial} computes the geometry of the monomial model explicitly and obtains the corresponding $L^4$ estimates as a specialization of the general theory. Finally, Section~\ref{sec:multiscale-frostman} develops the layer-cake representation, the distribution of the energy across scales, selection of concentration at one scale, and the Frostman regime; Subsection~\ref{sec:frostman} contains the Frostman energy bounds and the corresponding model constructions.

\subsection*{Notation and conventions}

We use the notation introduced above: $\mathcal N$ for the Gauss map defined by the unit normal, $T$ for the unit tangent, $\kappa$ for scalar curvature, $\mu_\kappa$ for the turning measure, and $\nu_f$ for the pushforward measure on normal directions. For a measure $\mu$ and a map $F$, $F_\#\mu$ denotes the pushforward of $\mu$ by $F$.

For nonnegative quantities, $A\lesssim_{\mathcal P}B$ means $A\le C_{\mathcal P}B$ for a constant depending only on the parameters indicated by $\mathcal P$; define $A\gtrsim_{\mathcal P}B$ analogously and $A\asymp_{\mathcal P}B$ when both inequalities hold. Generic constants $c,C>0$ may change from one occurrence to the next when there is no ambiguity. The symbol $\sim$ is reserved for combinatorial adjacency of caps: $I\sim J$.

We write $B(x,r)$ for Euclidean balls and $B_\Gamma(\xi,r)$ for intrinsic arc-length balls on $\Gamma$. The expression $A+B(0,r)$ denotes the Euclidean $r$-neighborhood of a set $A$. We introduce no global abbreviation for mass: localized masses are written explicitly once the corresponding localized functions have been introduced, and for an abstract angular measure we write $\nu(\Omega_k)$ directly. Local abbreviations are allowed within a proof when they are defined at the point of use.

\section{Geometry of the terminal scale}
\label{sec:terminal-geometry}

This section develops the geometric framework used in the $L^4$ estimates of the following sections. We begin with preliminaries on the sum map and the bilinear identity in the relevant local setting. We then fix the general geometric class, define the terminal scale through the turning measure, and construct adapted partitions. The structural metric--measure hypothesis is doubling; quantitative finite-type data enter only as additional information used to describe the scale locally and obtain sharper geometric exponents.

\subsection{Geometric and bilinear preliminaries}
\subsubsection{Conventions}

Let $\Gamma=\gamma(J)\subset\mathbb R^2$ be a regular $C^2$ curve. We retain the global conventions from the Introduction for $\mathrm d\sigma$, $T$, $\mathcal N$, and the nonnegative scalar curvature $\kappa$. When working with a graph $\gamma(t)=(t,\phi(t))$, we assume that the slope is bounded on the chart under consideration. Whenever no confusion can arise, we identify a function on the arc with its pullback by $\gamma$ and write $\mathrm d\sigma(t)=|\gamma'(t)|\,\mathrm dt$.
Implicit constants may depend on the declared structural parameters of the curve and the chart, but never on $R$ or $f$.

When a subarc $I$ is parametrized by arc length, we denote its midpoint by $t_I$. For $\delta>0$, define the tangential rectangle in frequency space by
\[
\Pi(I,\delta)
=
\left\{
\gamma(t_I)+aT(t_I)+b\mathcal N(t_I):
|a|\le 2|I|,\ |b|\le 2\delta
\right\}.
\]
If $A\ge1$, we write $A\Pi(I,\delta)$ for the concentric dilation by the factor $A$. The fixed numerical factors in this definition are immaterial: changing them only modifies structural constants in the inclusions below.

\subsubsection{Coordinates for the sum map}

For ordered pairs $u<v$, introduce $\bar t=(u+v)/2$ and $\tilde t=(v-u)/2$, always with $\bar t\pm\tilde t$ inside the chart. If $\gamma(t)=(t,\phi(t))$, define
\begin{equation}\label{eq:def-qphi}
q_\phi(\bar t,\tilde t)=\phi(\bar t+\tilde t)+\phi(\bar t-\tilde t)-2\phi(\bar t).
\end{equation}
Then
\begin{equation}\label{eq:sum-map-center-distance}
\gamma(\bar t-\tilde t)+\gamma(\bar t+\tilde t)
=
\bigl(2\bar t,2\phi(\bar t)+q_\phi(\bar t,\tilde t)\bigr),
\end{equation}
and
\begin{equation}\label{eq:qphi-separation-derivative}
\partial_{\tilde t}q_\phi(\bar t,\tilde t)
=
\phi'(\bar t+\tilde t)-\phi'(\bar t-\tilde t).
\end{equation}
In particular, if $S(\bar t,\tilde t)=\gamma(\bar t-\tilde t)+\gamma(\bar t+\tilde t)$, then
\[
|\det DS(\bar t,\tilde t)|=2\,|\partial_{\tilde t}q_\phi(\bar t,\tilde t)|.
\]
Thus $q_\phi$ measures the transverse displacement of the sum map relative to the diagonal $\tilde t=0$, whereas $\partial_{\tilde t}q_\phi$ measures its transversality in the coordinates $(\bar t,\tilde t)$.

\begin{lemma}[Nondegenerate case]\label{lem:q-nondegenerate}
Assume $0<c_0\le\phi''\le C_0$ on the chart. Whenever $[\bar t-\tilde t,\bar t+\tilde t]$ is contained in it,
\[
q_\phi(\bar t,\tilde t)\asymp_{c_0,C_0}\tilde t^2,
\qquad
\partial_{\tilde t}q_\phi(\bar t,\tilde t)\asymp_{c_0,C_0}\tilde t.
\]
\end{lemma}

\begin{proof}
By integration,
\[
\partial_{\tilde t}q_\phi(\bar t,\tilde t)
=\int_{\bar t-\tilde t}^{\bar t+\tilde t}\phi''(u)\,\mathrm du,
\]
so $2c_0\tilde t\le\partial_{\tilde t}q_\phi(\bar t,\tilde t)\le2C_0\tilde t$. Integrating once more in $\tilde t$ and using $q_\phi(\bar t,0)=0$ yields the estimate for $q_\phi$.
\end{proof}

\subsubsection{Exact bilinear identity}

The following identity is the planar form of the bilinear change of variables that expresses transversality through the normal directions. This change-of-variables argument is standard in the bilinear restriction approach and appears in related forms in Tao--Vargas--Vega and Bennett--Nakamura--Shiraki~\cite{TaoVargasVega1998,BennettNakamuraShiraki2024}. We record it with our normalizations because it is the basic identity used to treat transverse pairs.

\begin{proposition}[Bilinear identity]\label{prop:bilinear-identity}
Let $I_1,I_2$ be disjoint ordered parameter intervals, and let $f_j$ be functions supported on the corresponding arcs. Assume that the sum map
\[
(u,v)\longmapsto \gamma(u)+\gamma(v)
\]
is injective on $I_1\times I_2$. Then, with both sides understood in $[0,\infty]$,
\begin{equation}\label{eq:bilinear-identity}
\|E_\Gamma f_1E_\Gamma f_2\|_{L^2(\mathbb R^2)}^2
=
\iint_{I_1\times I_2}
\frac{|f_1(u)|^2|f_2(v)|^2}
{|\det(\mathcal N(u),\mathcal N(v))|}
\,\mathrm d\sigma(u)\mathrm d\sigma(v).
\end{equation}
\end{proposition}

\begin{proof}
The Fourier transform of $E_\Gamma f_1E_\Gamma f_2$ is the pushforward of
\[
f_1(u)f_2(v)\,\mathrm d\sigma(u)\mathrm d\sigma(v)
\]
under the sum map. The injectivity hypothesis allows us to apply the area formula without multiplicity. The Euclidean Jacobian of the sum map is
\[
|\det(\gamma'(u),\gamma'(v))|.
\]
Since $\gamma'=|\gamma'|T$ and
\[
|\det(T(u),T(v))|=|\det(\mathcal N(u),\mathcal N(v))|,
\]
the factors $|\gamma'|$ cancel against the two arc-length measures. Plancherel gives \eqref{eq:bilinear-identity}. If the Jacobian vanishes, the identity is understood in the extended sense by restricting first to regions where the Jacobian is bounded below and then passing to the limit.
\end{proof}

\begin{remark}\label{rem:det-normal-distance}
Assume that the angular range of a chart is at most $\theta_0<\pi$. Then, for $u,v$ in that chart,
\[
|\det(\mathcal N(u),\mathcal N(v))|\asymp_{\theta_0}|\mathcal N(u)-\mathcal N(v)|.
\]
Indeed, if $\vartheta$ is the angle between the normals, the two sides are respectively $|\sin\vartheta|$ and $2|\sin(\vartheta/2)|$. This comparability is used only on charts whose angular range is structurally controlled.
\end{remark}

\subsection{General geometric hypotheses}

From this point on, let $\Gamma$ be a compact convex $C^2$ curve with no affine subarcs. Fix a finite cover by convex graph charts of angular range less than $\theta_0<\pi/2$ and uniformly bounded slope, with each working chart lying inside a slightly larger chart whenever an estimate requires room. The terminal scale is defined globally on $\Gamma$ before local coordinates are introduced. The $C^2$ hypothesis suffices for the geometric machinery underlying the $L^4$ estimates based on the turning measure; statements that infer finite-type data automatically for a fixed curve use the additional $C^\infty$ hypothesis when needed.

\subsection{Turning measure and terminal scale}

Let $d_\Gamma$ denote intrinsic arc-length distance on $\Gamma$, and for $\xi\in\Gamma$ and $r>0$ write
\[
B_\Gamma(\xi,r)
=
\{\eta\in\Gamma:d_\Gamma(\xi,\eta)<r\}.
\]
On a compact arc this definition automatically incorporates truncation at the endpoints; on a closed curve we use geodesic arc-length distance. We use the turning measure $\mu_\kappa$ introduced in the Introduction. The angular variation of a planar curve is classically measured by integrating curvature with respect to arc length; equivalently, it is the length of the image under the Gauss map (see, for instance, do Carmo~\cite{doCarmo2016}). For every subarc $A$ with total turning less than $\pi$,
\[
\mu_\kappa(A)=\int_A\kappa\,\mathrm d\sigma
=\operatorname{length}_{S^1}(\mathcal N(A)).
\]
Choose a structural radius $r_*>0$ so small that every ball $B_\Gamma(\xi,4r_*)$ is contained in one of the enlarged charts of the fixed cover. We use the standard notion of a doubling measure, restricted here to the local structural range; for the corresponding metric--measure theory see Heinonen~\cite{Heinonen2001}. We assume that $\mu_\kappa$ is doubling on this range: there exists $D\ge1$ such that
\begin{equation}\label{eq:doubling}
\mu_\kappa(B_\Gamma(\xi,2r))
\le
D\,\mu_\kappa(B_\Gamma(\xi,r))
\end{equation}
for every $\xi\in\Gamma$ and every $r>0$ with $2r\le4r_*$. We adopt the following convention for constants. The notation $\lesssim_\Gamma$, $\gtrsim_\Gamma$, and $\asymp_\Gamma$ may depend only on the fixed geometric data of the curve and the chart cover (including $\theta_0$, slope bounds, margins, and the number of charts), but does not absorb the doubling constant $D$. Quantitative dependence on doubling is displayed explicitly through $D$; finite-type data are indicated by $\mathfrak T$, and auxiliary parameters such as $A$ or $\iota$ are shown when they enter. Refining the working cover if necessary, we also assume that every subchart on which the bilinear arguments are applied has intrinsic length at most $r_*$; the enlarged charts retain the fixed margin used for later dilations.

Since $\Gamma$ contains no affine subarcs, every nontrivial intrinsic ball has positive turning mass. Hence, for each $\xi\in\Gamma$, the function
\[
r\longmapsto r\,\mu_\kappa(B_\Gamma(\xi,r)),
\qquad 0\le r\le r_*,
\]
is continuous and strictly increasing on $(0,r_*]$. Moreover, by compactness, the function
\[
\xi\longmapsto r_*\mu_\kappa(B_\Gamma(\xi,r_*))
\]
has a positive minimum. Thus there exists $R_*=R_*(\Gamma)\ge1$ such that, for every $R\ge R_*$ and every $\xi\in\Gamma$, the equation
\[
r\,\mu_\kappa(B_\Gamma(\xi,r))=R^{-1}
\]
has a unique solution $r=r_R(\xi)\in(0,r_*]$. For $1\le R<R_*$ we set
\[
r_R(\xi)=r_{R_*}(\xi).
\]
With this convention, uniformly for $R\ge1$ and $\xi\in\Gamma$,
\begin{equation}\label{eq:terminal-general}
r_R(\xi)\,
\mu_\kappa(B_\Gamma(\xi,r_R(\xi)))
\asymp_\Gamma R^{-1},
\end{equation}
and define globally
\begin{equation}\label{eq:rho-general}
\rho_R(\xi)=\frac{R^{-1}}{r_R(\xi)}.
\end{equation}
For $R\ge R_*$ the relation \eqref{eq:terminal-general} is an equality. In particular, in this range
\[
\rho_R(\xi)=\mu_\kappa(B_\Gamma(\xi,r_R(\xi))),
\]
so $R\mapsto r_R(\xi)$ and $R\mapsto\rho_R(\xi)$ are nonincreasing; with the convention above, the same monotonicity holds for all $R\ge1$. If
\[
I_R(\xi)=B_\Gamma(\xi,r_R(\xi)),
\]
then $|I_R(\xi)|\asymp r_R(\xi)$, also when the ball is truncated by an endpoint of $\Gamma$. Since $I_R(\xi)$ lies in a chart of fixed angular range,
\[
\operatorname{diam}\mathcal N(I_R(\xi))\asymp_\Gamma\mu_\kappa(I_R(\xi)),
\]
and by \eqref{eq:terminal-general},
\[
|I_R(\xi)|\operatorname{diam}\mathcal N(I_R(\xi))\asymp_\Gamma R^{-1}.
\]
Thus the terminal product condition is an intrinsic consequence of the global definition. When $\xi=\gamma(t)$ in an arc-length parametrized chart, we abbreviate $r_R(\gamma(t))$ and $\rho_R(\gamma(t))$ by $r_R(t)$ and $\rho_R(t)$. This local notation always denotes the global functions just defined.

\begin{lemma}[Local comparisons under doubling]\label{lem:doubling-local-comparison}
Let $\mu_\kappa$ be doubling with constant $D$ on the structural range fixed above.

\begin{enumerate}
\item For fixed $A\ge0$ and $\lambda>0$, there exists $C=C(A,\lambda,D)$ such that, if
\[
d_\Gamma(\xi,\xi')\le A r
\]
and every ball arising when the radii are dilated up to
\[
\max\{A+1,A+\lambda,\lambda,1\}\,r
\]
remains in the range of \eqref{eq:doubling}, then
\[
C^{-1}\mu_\kappa(B_\Gamma(\xi,r))
\le
\mu_\kappa(B_\Gamma(\xi',\lambda r))
\le
C\mu_\kappa(B_\Gamma(\xi,r)).
\]

\item For each $c\in(0,1]$ there exists $c_{\mathrm{dbl}}=c_{\mathrm{dbl}}(c,D)>0$ with the following property. If $J'\subset J$ are subarcs contained in a working chart, $|J|\le r_*$, and
\[
|J'|\ge c|J|,
\]
then
\[
\mu_\kappa(J')\ge c_{\mathrm{dbl}}\mu_\kappa(J).
\]
\end{enumerate}
\end{lemma}

\begin{proof}
For the first assertion we use the inclusions
\[
B_\Gamma(\xi,r)
\subset
B_\Gamma(\xi',(A+1)r)
\]
and
\[
B_\Gamma(\xi',\lambda r)
\subset
B_\Gamma(\xi,(A+\lambda)r).
\]
Comparing each larger radius with the corresponding smaller radius through \eqref{eq:doubling}, the required number of iterations is bounded by a constant depending only on $A$ and $\lambda$. This gives both inequalities.

For the second assertion, let $y$ be the intrinsic midpoint of $J'$. Then, apart from endpoints, which do not affect the measure,
\[
B_\Gamma\!\left(y,\frac{|J'|}{2}\right)\subset J',
\]
whereas
\[
J\subset B_\Gamma(y,|J|).
\]
Since $|J|\le r_*$ and
\[
\frac{|J|}{|J'|/2}\le\frac{2}{c},
\]
all dilations remain in the structural range. Iterating \eqref{eq:doubling} from $|J'|/2$ to $|J|$ gives
\[
\mu_\kappa(J)
\lesssim_{c,D}
\mu_\kappa\!\left(B_\Gamma\!\left(y,\frac{|J'|}{2}\right)\right)
\le
\mu_\kappa(J'),
\]
which proves the assertion.
\end{proof}

\begin{lemma}[Slow variation]\label{lem:slow-variation}
For each $A\ge1$ there exists $C=C(\Gamma,D,A)$ such that, if
\[
d_\Gamma(\xi,\xi')\le A r_R(\xi),
\]
then
\[
C^{-1}r_R(\xi)\le r_R(\xi')\le Cr_R(\xi)
\]
and
\[
C^{-1}\rho_R(\xi)\le\rho_R(\xi')\le C\rho_R(\xi).
\]
\end{lemma}

\begin{proof}
Set
\[
r=r_R(\xi),
\qquad
r'=r_R(\xi').
\]
Applying the same compactness argument used to define $R_*$ at the fixed radius $r_*/(A+1)$, there exists $R_A=R_A(A,\Gamma)$ such that, for $R\ge R_A$,
\[
(A+1)r_R(\zeta)\le r_*
\qquad
\text{for every }\zeta\in\Gamma.
\]
In this range \eqref{eq:terminal-general} is an equality. The inclusion
\[
B_\Gamma(\xi,r)
\subset
B_\Gamma(\xi',(A+1)r)
\]
gives
\[
(A+1)r\,\mu_\kappa(B_\Gamma(\xi',(A+1)r))
\ge
r\,\mu_\kappa(B_\Gamma(\xi,r))
=R^{-1}.
\]
Since $t\mapsto t\mu_\kappa(B_\Gamma(\xi',t))$ is increasing, it follows that
\[
r'\le(A+1)r.
\]

For the reverse bound, set
\[
k_A=\left\lceil\log_2(A+1)\right\rceil,
\qquad
c_A=D^{-k_A}.
\]
Since $c_A\le1$,
\[
B_\Gamma(\xi',c_A r)
\subset
B_\Gamma(\xi,(A+c_A)r)
\subset
B_\Gamma(\xi,(A+1)r).
\]
Iterating \eqref{eq:doubling} at most $k_A$ times,
\[
\mu_\kappa(B_\Gamma(\xi',c_A r))
\le
D^{k_A}\mu_\kappa(B_\Gamma(\xi,r)).
\]
Therefore
\[
c_A r\,\mu_\kappa(B_\Gamma(\xi',c_A r))
\le
c_A D^{k_A}R^{-1}
=R^{-1},
\]
and monotonicity of the terminal function gives
\[
r'\ge c_A r.
\]
This proves the comparability of $r_R$ for $R\ge R_A$.

On the bounded interval $1\le R<R_A$, we always have $r_R(\zeta)\le r_*$ and, by \eqref{eq:terminal-general} and the bound $\mu_\kappa(B_\Gamma(\zeta,r_R(\zeta)))\le\mu_\kappa(\Gamma)$,
\[
r_R(\zeta)\gtrsim_\Gamma R_A^{-1}.
\]
Thus the terminal radii are uniformly comparable in this range, and the constant is absorbed into $C(\Gamma,D,A)$. Finally,
\[
\frac{\rho_R(\xi')}{\rho_R(\xi)}
=
\frac{r_R(\xi)}{r_R(\xi')}
\]
by \eqref{eq:rho-general}, which yields the assertion for $\rho_R$.
\end{proof}

\subsection{Linearization and the product condition}

For a subarc $I$ of length $L$, write
\[
\alpha(I)=\operatorname{diam}_{\mathbb R^2}\mathcal N(I).
\]
If the total turning of $I$ is $\vartheta(I)=\mu_\kappa(I)<\pi$, then
\[
\alpha(I)=2\sin\!\left(\frac{\vartheta(I)}2\right).
\]
On the charts of structurally bounded angular range used here,
\[
\alpha(I)\asymp_\Gamma\mu_\kappa(I).
\]
The natural terminal condition is
\begin{equation}\label{eq:product-terminal}
L\alpha(I)\asymp R^{-1}.
\end{equation}

\begin{lemma}[Product implies linearization]\label{lem:product-linearization}
Let $t_0\in I$. Then the distance from $\gamma(I)$ to the tangent line at $t_0$ satisfies
\[
\sup_{t\in I}\operatorname{dist}(\gamma(t),\gamma(t_0)+\mathbb RT(t_0))
\lesssim_\Gamma L\alpha(I).
\]
In particular, \eqref{eq:product-terminal} implies that $\gamma(I)$ fits inside a tangential rectangle of length $O(L)$ and normal thickness $O(R^{-1})$.
\end{lemma}

\begin{proof}
Since
\[
\gamma(t)-\gamma(t_0)=\int_{t_0}^tT(u)\,\mathrm du,
\]
and the angular variation of $T$ on $I$ is $O(\alpha(I))$, the normal component of the integral is bounded by $L\alpha(I)$.
\end{proof}

The converse implication is false without a nonconcentration hypothesis on curvature: almost all of the turning may concentrate on an arbitrarily small fraction near one endpoint of the interval. Under doubling, however, the equivalence is recovered.

\begin{proposition}[Linearization under doubling]\label{prop:linearization-doubling}
Assume that $I=[a,b]$ is a subarc parametrized by arc length, $|I|\le r_*$, that its total turning is smaller than a fixed constant $<\pi/2$, and that $\mu_\kappa$ is doubling. Then the normal deviation of the endpoint from the initial tangent satisfies
\[
\operatorname{dist}(\gamma(b),\gamma(a)+\mathbb RT(a))
\asymp_{\Gamma,D}
|I|\mu_\kappa(I).
\]
\end{proposition}

\begin{proof}
Write $L=|I|$ and
\[
\Theta(r)=\mu_\kappa([a,a+r]),
\qquad 0\le r\le L.
\]
Since the total turning is smaller than a fixed constant $<\pi/2$,
\[
\int_0^L\sin\Theta(r)\,\mathrm dr
\asymp_\Gamma
\int_0^L\Theta(r)\,\mathrm dr.
\]
The upper bound is immediate:
\[
\int_0^L\Theta(r)\,\mathrm dr
\le
L\mu_\kappa(I).
\]
For the lower bound, consider the initial half
\[
I_-=[a,a+L/2].
\]
By the second part of Lemma~\ref{lem:doubling-local-comparison}, applied with $c=1/2$,
\[
\mu_\kappa(I_-)
\gtrsim_D
\mu_\kappa(I).
\]
Since $\Theta$ is increasing, for every $r\in[L/2,L]$,
\[
\Theta(r)
\ge
\mu_\kappa(I_-)
\gtrsim_D
\mu_\kappa(I).
\]
Therefore
\[
\int_0^L\Theta(r)\,\mathrm dr
\ge
\int_{L/2}^L\Theta(r)\,\mathrm dr
\gtrsim_D
L\mu_\kappa(I),
\]
which proves the equivalence.
\end{proof}

The preceding proposition gives a precise meaning to the term \emph{terminal}. For $\xi\in\Gamma$ and $0<r\le r_*$, define the centered linearization defect by
\[
\mathfrak e_\Gamma(\xi,r)
=
\sup_{\eta\in B_\Gamma(\xi,r)}
\operatorname{dist}\bigl(\eta,\xi+\mathbb RT(\xi)\bigr).
\]

\begin{proposition}[Tangential maximality of the terminal scale]\label{prop:maximal-tangential-scale}
Assume that $\mu_\kappa$ is doubling on the structural range. Then, uniformly for $\xi\in\Gamma$ and $0<r\le r_*$,
\[
\mathfrak e_\Gamma(\xi,r)
\asymp_{\Gamma,D}
r\,\mu_\kappa(B_\Gamma(\xi,r)).
\]
For $A>0$ and $R\ge R_*(\Gamma)$ define
\[
r_{\mathrm{lin},R}^{(A)}(\xi)
=
\sup\left\{
0<r\le r_*:
\mathfrak e_\Gamma(\xi,r)\le A R^{-1}
\right\}.
\]
Then
\[
r_{\mathrm{lin},R}^{(A)}(\xi)
\asymp_{\Gamma,D,A}
r_R(\xi).
\]
In particular, up to structural constants, $r_R(\xi)$ is the largest tangential scale centered at $\xi$ on which the curve remains linearizable with normal error $O(R^{-1})$.
\end{proposition}

\begin{proof}
Let $I=B_\Gamma(\xi,r)$. Since $|I|\asymp r$ and its total turning lies in the structural range, Lemma~\ref{lem:product-linearization} gives
\[
\mathfrak e_\Gamma(\xi,r)
\lesssim_\Gamma
r\,\mu_\kappa(I).
\]
For the reverse bound, choose one of the two subarcs of $I$ issuing from $\xi$ whose length is at least a fixed fraction of $|I|$; at a global endpoint, simply take the available one-sided subarc. Denote this subarc by $J$. Then $|J|\gtrsim r$. If $y$ is its intrinsic midpoint, there is a structural constant $c>0$ such that
\[
B_\Gamma(y,cr)\subset J,
\qquad
d_\Gamma(y,\xi)\le r.
\]
The first part of Lemma~\ref{lem:doubling-local-comparison}, with fixed ratios of radii, therefore gives
\[
\mu_\kappa(J)
\ge
\mu_\kappa(B_\Gamma(y,cr))
\gtrsim_D
\mu_\kappa(B_\Gamma(\xi,r))
=
\mu_\kappa(I).
\]
Applying Proposition~\ref{prop:linearization-doubling} to $J$, parametrized from $\xi$ toward its other endpoint (reversing orientation if necessary), we obtain a point $\eta_J\in J$ such that
\[
\operatorname{dist}\bigl(\eta_J,\xi+\mathbb RT(\xi)\bigr)
\gtrsim_{\Gamma,D}
|J|\mu_\kappa(J)
\gtrsim_{\Gamma,D}
r\mu_\kappa(I).
\]
This proves the first equivalence.

For $R\ge R_*(\Gamma)$,
\[
r_R(\xi)\,\mu_\kappa(B_\Gamma(\xi,r_R(\xi)))=R^{-1}.
\]
If $r$ satisfies $\mathfrak e_\Gamma(\xi,r)\le A R^{-1}$, the equivalence just proved implies
\[
r\,\mu_\kappa(B_\Gamma(\xi,r))\lesssim_{\Gamma,D} A R^{-1}.
\]
If in addition $r\ge r_R(\xi)$, monotonicity of $r\mapsto\mu_\kappa(B_\Gamma(\xi,r))$ gives
\[
r\,\mu_\kappa(B_\Gamma(\xi,r))
\ge
\frac{r}{r_R(\xi)}\,r_R(\xi)\mu_\kappa(B_\Gamma(\xi,r_R(\xi)))
=
\frac{r}{r_R(\xi)}R^{-1},
\]
and hence $r\lesssim_{\Gamma,D,A}r_R(\xi)$. This gives the upper bound for $r_{\mathrm{lin},R}^{(A)}(\xi)$.

Conversely, if $0<c\le1$, then
\[
c r_R(\xi)\,\mu_\kappa(B_\Gamma(\xi,c r_R(\xi)))
\le
c\,r_R(\xi)\mu_\kappa(B_\Gamma(\xi,r_R(\xi)))
=
cR^{-1}.
\]
The upper bound for the defect shows that, choosing $c=c(\Gamma,A)>0$ sufficiently small,
\[
\mathfrak e_\Gamma(\xi,c r_R(\xi))\le A R^{-1}.
\]
Thus
\[
r_{\mathrm{lin},R}^{(A)}(\xi)
\gtrsim_{\Gamma,A}
r_R(\xi),
\]
and the comparability follows.
\end{proof}

\begin{remark}[Comparison with standard scales]
In a nondegenerate region,
\[
\mu_\kappa(B_\Gamma(\xi,r))\asymp r.
\]
Hence the preceding characterization recovers the parabolic scale
\[
r_R(\xi)\asymp R^{-1/2}.
\]
If $\xi=\zeta$ is a flat point of order $m$, then
\[
\mu_\kappa(B_\Gamma(\zeta,r))\asymp r^{m+1}
\]
and therefore
\[
r_R(\zeta)\asymp R^{-1/(m+2)}.
\]
In particular, when $m>0$, the uniform scale $R^{-1/2}$ remains linearizable but is strictly smaller than the terminal scale: it partitions the curve more finely than necessary. For the monomial $\gamma_k(t)=(t,t^k)$, one has $m=k-2$ and recovers
\[
r_R(0)\asymp R^{-1/k},
\]
comparable, up to constants, to the maximal linearization scale of Schippa's canonical monomial covering for the integer exponents treated there~\cite{Schippa2024}. For real $k\ge3$, the same comparabilities follow directly from $\kappa(t)\asymp_k t^{k-2}$ and the intrinsic definition of the terminal scale. This tangential maximality is distinct from the maximality of the angular cutoff in the energy kernel, treated separately in the general $L^4$ estimates.
\end{remark}

\subsection{Finite type}
\label{sec:finite-type}

The qualitative notion of a finite-type curve is standard in restriction theory for degenerate curves; see Christ, Drury--Marshall, and Dendrinos--Müller~\cite{Christ1985,DruryMarshall1987,DendrinosMuller2013}. In our arc-length parametrization, $m_j$ denotes the vanishing order of the curvature, so that in the monomial graph model the corresponding order of contact is $m_j+2$. Here we fix quantitative finite-type data, specific to this paper, that record the constants needed uniformly over a family.

\begin{definition}[Quantitative finite-type data]\label{def:quantitative-finite-type}
We say that $\Gamma$ is equipped with \emph{quantitative finite-type data}
\[
\mathfrak T=
(n_*,m_*,r_{\mathrm{ft}},A_{\mathrm{ft}}),
\qquad A_{\mathrm{ft}}\ge1,
\]
if $n_*,m_*$ are nonnegative integers,
\[
8r_{\mathrm{ft}}\le |\Gamma|,
\]
the zero set of the curvature is
\[
Z=\{\zeta_1,\dots,\zeta_{n_{\mathrm{flat}}}\},
\qquad n_{\mathrm{flat}}\le n_*,
\]
the intrinsic balls $B_\Gamma(\zeta_j,4r_{\mathrm{ft}})$ are pairwise disjoint, and for each $j$ there exists an integer $1\le m_j\le m_*$ such that
\[
A_{\mathrm{ft}}^{-1}d_\Gamma(\xi,\zeta_j)^{m_j}
\le
\kappa(\xi)
\le
A_{\mathrm{ft}}d_\Gamma(\xi,\zeta_j)^{m_j},
\qquad
\xi\in B_\Gamma(\zeta_j,4r_{\mathrm{ft}}).
\]
Away from the flat neighborhoods we require
\[
A_{\mathrm{ft}}^{-1}
\le
\kappa(\xi)
\le
A_{\mathrm{ft}},
\qquad
\xi\in
\Gamma\setminus\bigcup_{j=1}^{n_{\mathrm{flat}}} B_\Gamma(\zeta_j,r_{\mathrm{ft}}).
\]
If $Z=\varnothing$, the conditions involving the orders $m_j$ are vacuous; in particular, $m_*$ may be retained as a uniform bound for a family and need not coincide with an order actually attained by that curve. For a family of curves, \emph{quantitatively finite type} means that the same data $\mathfrak T$, together with the structural chart controls fixed at the beginning of the section, apply to every member of the family.
\end{definition}

For a fixed curve, define its actual maximal order by
\[
m_{\max}(\Gamma)
=
\begin{cases}
\max_{1\le j\le n_{\mathrm{flat}}}m_j,& Z\neq\varnothing,\\
0,& Z=\varnothing.
\end{cases}
\]
This number depends on the curve, whereas $m_*$ in $\mathfrak T$ is only an available upper bound for the orders in a family. For a fixed $C^\infty$ curve whose curvature zeros all have finite order, the existence of quantitative data is automatic. Indeed, Taylor's theorem gives comparability with a power near each zero; the zeros are isolated and, by compactness, there are only finitely many, so one may choose disjoint neighborhoods, decrease $r_{\mathrm{ft}}$ so that also $8r_{\mathrm{ft}}\le|\Gamma|$, and use the positive minimum of the curvature on the complement. Moreover, for that curve the data may be chosen with
\[
m_*=m_{\max}(\Gamma).
\]
The quantitative formulation in Definition~\ref{def:quantitative-finite-type} records which parameters must remain uniform when families are considered.

\begin{proposition}[Turning measure near a flat point]\label{prop:finite-type-doubling}
Assume that $\zeta_j$ is a zero described by Definition~\ref{def:quantitative-finite-type}, of order $m_j\ge1$. If
\[
\xi\in B_\Gamma(\zeta_j,2r_{\mathrm{ft}}),
\qquad
0<r\le r_{\mathrm{ft}},
\]
then
\[
\mu_\kappa(B_\Gamma(\xi,r))
\asymp_{\mathfrak T}
r\bigl(d_\Gamma(\xi,\zeta_j)+r\bigr)^{m_j}.
\]
In particular, for $0<r\le r_{\mathrm{ft}}/2$,
\[
\mu_\kappa(B_\Gamma(\xi,2r))
\lesssim_{\mathfrak T}
\mu_\kappa(B_\Gamma(\xi,r)),
\]
whenever $\xi\in B_\Gamma(\zeta_j,2r_{\mathrm{ft}})$.
\end{proposition}

\begin{proof}
Parametrize by arc length and write
\[
a=d_\Gamma(\xi,\zeta_j).
\]
The condition $8r_{\mathrm{ft}}\le|\Gamma|$ ensures that, for $0<r\le r_{\mathrm{ft}}$, every intrinsic ball $B_\Gamma(\xi,r)$ has arc length comparable to $r$, even if it is truncated by a global endpoint. Moreover, if $\xi\in B_\Gamma(\zeta_j,2r_{\mathrm{ft}})$, then
\[
B_\Gamma(\xi,r)\subset B_\Gamma(\zeta_j,3r_{\mathrm{ft}}),
\]
so the entire ball remains in the neighborhood where the quantitative curvature bounds hold. Therefore the mass of the ball is comparable to the integral of $|u|^{m_j}$ over an interval of length comparable to $r$ centered at distance $a$ from the origin; if $\zeta_j$ or the ball itself lies near an endpoint of $\Gamma$, that interval is truncated on one side. In all cases the integral satisfies
\[
\int_{B_\Gamma(\xi,r)}d_\Gamma(\eta,\zeta_j)^{m_j}\,\mathrm d\sigma(\eta)
\asymp_{m_j}
r(a+r)^{m_j}.
\]
The assertion follows from the upper and lower bounds for $\kappa$. Replacing $r$ by $2r$ yields doubling, as long as the balls remain in the indicated structural neighborhood.
\end{proof}

The explicit formulas for the terminal scale deduced from this proposition are local. Fix a zero $\zeta_j=\gamma(t_0)$ of order $m=m_j$. By compactness and positivity of the turning mass of every nontrivial intrinsic ball, there exists
\[
R_{\mathrm{ft}}=R_{\mathrm{ft}}(\Gamma,\mathfrak T)\ge R_*(\Gamma)
\]
such that
\[
r_R(\xi)\le r_{\mathrm{ft}}
\qquad
(R\ge R_{\mathrm{ft}},\ \xi\in\Gamma).
\]
We enlarge this same threshold, if necessary, after fixing the radius $r_{\mathrm{lg}}$ from Lemma~\ref{lem:finite-type-lower-growth}; no second threshold is introduced for the finite-type regime. In particular, if
\[
\xi=\gamma(t)\in B_\Gamma(\zeta_j,2r_{\mathrm{ft}}),
\qquad
R\ge R_{\mathrm{ft}},
\]
the terminal ball remains in the structural range of Proposition~\ref{prop:finite-type-doubling}. Solving \eqref{eq:terminal-general} then gives, with constants depending only on the indicated structural data,
\begin{equation}\label{eq:ell-finite-type}
r_R(t)
\asymp_{\Gamma,\mathfrak T}
R^{-1/2}
\left(|t-t_0|+R^{-1/(m+2)}\right)^{-m/2},
\end{equation}
and
\begin{equation}\label{eq:rho-finite-type}
\rho_R(t)
\asymp_{\Gamma,\mathfrak T}
R^{-1/2}
\left(|t-t_0|+R^{-1/(m+2)}\right)^{m/2}.
\end{equation}
Here $t$ denotes local arc length. In particular, for $R\ge R_{\mathrm{ft}}$,
\[
r_R(t_0)\asymp_{\Gamma,\mathfrak T} R^{-1/(m+2)},
\qquad
\rho_R(t_0)\asymp_{\Gamma,\mathfrak T} R^{-(m+1)/(m+2)}.
\]
These formulas are local to the quantitative neighborhood of the flat point. Global estimates over the bounded range $1\le R<R_{\mathrm{ft}}$ use the terminal-scale convention from the beginning of the section, with that compact interval absorbed into structural constants.

For integer $k\ge3$, the model $\gamma_k(t)=(t,t^k)$ fits the preceding finite-type framework, with $m=k-2$ at $t=0$ and $\mathrm d\sigma\asymp_k\mathrm dt$ on $[0,1]$; in the local regime, the preceding formulas agree, up to constants depending on $k$, with the explicit scales $\ell_R$ and $\tau_R$. For real $k\ge3$ not necessarily integer, the comparability $\kappa(t)\asymp_k t^{k-2}$ yields the same scales directly from the terminal equation, while $k-2$ is not interpreted as one of the integer orders $m_j$ from Definition~\ref{def:quantitative-finite-type}. Section~\ref{sec:monomial} defines the monomial formulas globally on $[0,1]$, independently of this local reduction.

\subsection{Lower growth of the turning measure}

The doubling hypothesis already contains a quantitative form of lower growth.

\begin{lemma}[Lower growth from doubling]\label{lem:doubling-lower-growth}
Let $D$ be a doubling constant for $\mu_\kappa$ and set
\[
Q_D=\log_2D.
\]
Then, for $\xi\in\Gamma$ and $0<r'\le r$ in the structural range,
\[
\frac{\mu_\kappa(B_\Gamma(\xi,r'))}{\mu_\kappa(B_\Gamma(\xi,r))}
\ge
D^{-1}\left(\frac{r'}r\right)^{Q_D}.
\]
\end{lemma}

\begin{proof}
Let
\[
n=\left\lceil\log_2\frac r{r'}\right\rceil.
\]
Then $r\le2^nr'$. By monotonicity and iteration of doubling,
\[
\mu_\kappa(B_\Gamma(\xi,r))
\le
\mu_\kappa(B_\Gamma(\xi,2^nr'))
\le
D^n\mu_\kappa(B_\Gamma(\xi,r')).
\]
Since $n<\log_2(r/r')+1$,
\[
D^{-n}
\ge
D^{-1}\left(\frac{r'}r\right)^{\log_2D},
\]
and the conclusion follows.
\end{proof}

\begin{remark}[Doubling, finite order, and an infinitely flat example]\label{rem:infinite-flat-nondoubling}
For a convex $C^\infty$ curve with no affine subarcs, doubling rules out curvature zeros of infinite order. Indeed, Lemma~\ref{lem:doubling-lower-growth}, applied relative to a fixed structural radius, provides a polynomial lower bound for $\mu_\kappa(B_\Gamma(\gamma(t_0),r))$ as $r\downarrow0$. If $\kappa$ vanished to infinite order at $t_0$, smoothness would give $\mu_\kappa(B_\Gamma(\gamma(t_0),r))=O_q(r^q)$ for every $q$, contradicting that lower bound by taking $q>Q_D$. Hence every zero has finite order and is therefore isolated. If there were infinitely many zeros on the compact curve, they would have an accumulation point; smoothness would force infinite-order vanishing there, which has just been ruled out. Thus the zero set is finite. In particular, by Taylor's theorem and compactness, a fixed curve under these hypotheses admits quantitative finite-type data in the sense of Definition~\ref{def:quantitative-finite-type}. For families, uniformity of those data is additional information and must be stated.

The failure of doubling under infinite flatness can be seen directly. On a symmetric interval about $0$, define
\[
\kappa(t)
=
\begin{cases}
e^{-1/t^2},&t\neq0,\\
0,&t=0,
\end{cases}
\]
and set
\[
\vartheta(t)=\int_0^t\kappa(s)\,\mathrm ds,
\qquad
\gamma(t)=\gamma(0)+\int_0^t(\cos\vartheta(s),\sin\vartheta(s))\,\mathrm ds.
\]
Then $|\gamma'(t)|=1$ and the scalar curvature of $\gamma$ is precisely $\kappa(t)$. The curve is $C^\infty$ and convex, has no affine subarcs, and $t=0$ is a flat point of infinite order. For small $r>0$,
\[
\mu_\kappa(B_\Gamma(\gamma(0),r))
=2\int_0^r e^{-1/t^2}\,\mathrm dt.
\]
With the change of variables $u=1/t$,
\[
\int_0^r e^{-1/t^2}\,\mathrm dt
=
\int_{1/r}^{\infty}e^{-u^2}u^{-2}\,\mathrm du
\asymp r^3e^{-1/r^2}.
\]
The upper bound follows from $u^{-2}\le r^2$ and the elementary Gaussian-tail estimate; for the lower bound it is enough to integrate over $[r^{-1},r^{-1}+r]$. Consequently,
\[
\frac{\mu_\kappa(B_\Gamma(\gamma(0),2r))}
{\mu_\kappa(B_\Gamma(\gamma(0),r))}
\asymp
8e^{3/(4r^2)}
\longrightarrow\infty,
\qquad r\downarrow0.
\]
Thus the turning measure is not doubling.

This example clarifies the role of the doubling hypothesis. For a fixed $C^\infty$ curve it does not define a geometric class essentially different from finite type; rather, it isolates the metric--measure information used by the preceding arguments: local comparison of masses, slow variation, terminal scale, transverse control, and radius comparison. Quantitative finite-type data, by contrast, add uniform control over families and allow the geometric exponent to be identified from the vanishing orders.
\end{remark}

\begin{lemma}[Finite-type lower growth]\label{lem:finite-type-lower-growth}
Assume that $\Gamma$ has quantitative finite-type data $\mathfrak T$ in the sense of Definition~\ref{def:quantitative-finite-type}. Then there exists
\[
r_{\mathrm{lg}}=r_{\mathrm{lg}}(\mathfrak T)>0
\]
such that, for every $\xi\in\Gamma$ and $0<r'\le r\le r_{\mathrm{lg}}$,
\begin{equation}\label{eq:lower-growth}
\frac{\mu_\kappa(B_\Gamma(\xi,r'))}{\mu_\kappa(B_\Gamma(\xi,r))}
\gtrsim_{\mathfrak T}
\left(\frac{r'}r\right)^{m_*+1}.
\end{equation}
If $m_*=0$, the right-hand side is interpreted as $r'/r$.
\end{lemma}

\begin{proof}
Choose $r_{\mathrm{lg}}>0$ sufficiently small relative to $r_{\mathrm{ft}}$ and to the separation of the flat neighborhoods in Definition~\ref{def:quantitative-finite-type}. Let
\[
a=d_\Gamma(\xi,Z),
\]
with the convention $a=+\infty$ if $Z=\varnothing$.

First assume $a\ge2r$. If $a<2r_{\mathrm{ft}}$, separation of the flat neighborhoods determines a unique zero $\zeta_j$ at distance $a$ from $\xi$. For $\eta\in B_\Gamma(\xi,r)$,
\[
\frac a2\le d_\Gamma(\eta,\zeta_j)\le\frac{3a}{2},
\]
and the same holds on $B_\Gamma(\xi,r')$. Hence
\[
\mu_\kappa(B_\Gamma(\xi,t))\asymp_{\mathfrak T} t a^{m_j},
\qquad t\in\{r',r\}.
\]
If $a\ge2r_{\mathrm{ft}}$, both balls lie outside the neighborhoods $B_\Gamma(\zeta_j,r_{\mathrm{ft}})$, after imposing $r_{\mathrm{lg}}\le r_{\mathrm{ft}}/4$, and the curvature there is uniformly bounded above and below. Thus
\[
\mu_\kappa(B_\Gamma(\xi,t))\asymp_{\mathfrak T}t.
\]
In both cases,
\[
\frac{\mu_\kappa(B_\Gamma(\xi,r'))}{\mu_\kappa(B_\Gamma(\xi,r))}
\gtrsim_{\mathfrak T}\frac{r'}r
\ge
\left(\frac{r'}r\right)^{m_*+1}.
\]

Now assume $a<2r$. By the choice of $r_{\mathrm{lg}}$, there is a unique zero $\zeta_j$ at distance $a$ from $\xi$, and both balls lie in its quantitative neighborhood. Proposition~\ref{prop:finite-type-doubling} gives
\[
\frac{\mu_\kappa(B_\Gamma(\xi,r'))}{\mu_\kappa(B_\Gamma(\xi,r))}
\gtrsim_{\mathfrak T}
\frac{r'}r
\left(\frac{a+r'}{a+r}\right)^{m_j}.
\]
Since $0<r'\le r$,
\[
r(a+r')-r'(a+r)=a(r-r')\ge0,
\]
and therefore
\[
\frac{a+r'}{a+r}\ge\frac{r'}r.
\]
Thus
\[
\frac{\mu_\kappa(B_\Gamma(\xi,r'))}{\mu_\kappa(B_\Gamma(\xi,r))}
\gtrsim_{\mathfrak T}
\left(\frac{r'}r\right)^{m_j+1}
\ge
\left(\frac{r'}r\right)^{m_*+1}.
\]
The argument uses only intrinsic balls, so flat points located at an endpoint of an arc are included without modification.
\end{proof}

After increasing, if necessary, the previously fixed threshold $R_{\mathrm{ft}}=R_{\mathrm{ft}}(\Gamma,\mathfrak T)$, we assume from now on that
\[
r_R(\xi)\le\min\{r_{\mathrm{ft}},r_{\mathrm{lg}}\}
\qquad
(R\ge R_{\mathrm{ft}},\ \xi\in\Gamma).
\]
Thus a single threshold controls both the local finite-type formulas and the quantitative lower growth.

\subsection{Terminal partition}

The partition construction used later is collected in a single lemma. In addition to the terminal product, the result provides comparability of consecutive caps and angular separation of nonneighboring caps.

\begin{lemma}[Construction of the terminal partition]\label{lem:terminal-partition-general}
Fix one of the convex charts parametrized by arc length over a compact interval $J$. There exists a constant $n_{\mathrm{nb}}=n_{\mathrm{nb}}(D,\Gamma)\ge3$ such that, for every $R\ge1$, one can construct an ordered partition
\[
\mathcal P_R=\{I_j\}
\]
of $J$, with points $t_j\in I_j$, satisfying the following properties. If
\[
\alpha_j:=\operatorname{diam}\mathcal N(I_j),
\qquad
\rho_j:=\rho_R(\gamma(t_j)),
\]
then, for every $t\in I_j$,
\begin{equation}\label{eq:terminal-cap-general}
|I_j|\asymp_{\Gamma,D} r_R(\gamma(t_j))\asymp_{\Gamma,D}r_R(\gamma(t)),
\qquad
\alpha_j\asymp_{\Gamma,D}\rho_j\asymp_{\Gamma,D}\rho_R(\gamma(t)),
\qquad
|I_j|\alpha_j\asymp_{\Gamma,D}R^{-1}.
\end{equation}
Consecutive caps have comparable lengths and angular widths. Defining
\[
I_j\sim I_\ell
\quad\Longleftrightarrow\quad
|j-\ell|\le n_{\mathrm{nb}},
\]
each cap has $O_{\Gamma,D}(1)$ neighbors and, if $I_j\not\sim I_\ell$,
\begin{equation}\label{eq:angular-separation-general}
\operatorname{dist}(\mathcal N(I_j),\mathcal N(I_\ell))
\gtrsim_{\Gamma,D}
\alpha_j+\alpha_\ell.
\end{equation}
\end{lemma}

\begin{proof}
Write
\[
r(t)=r_R(\gamma(t)).
\]
First fix a constant $c>0$ sufficiently small in terms of the slow-variation constants for $A=1$. For each $t\in J$, consider in the arc-length coordinate of the chart the interval
\[
B(t)=(t-c\,r(t),t+c\,r(t)).
\]
The interval $B(t)$ need not be contained in $J$; only its center and its intersection with the chart are used. The elementary $5r$ covering lemma for intervals (see, for instance, Heinonen~\cite{Heinonen2001}) provides a subfamily with disjoint interiors
\[
B_j=B(t_j)
\]
such that the concentric intervals of radius $5c\,r(t_j)$ cover $J$. The family is finite: for $R\ge R_*$, the terminal identity and the bound $\mu_\kappa(B_\Gamma(\gamma(t),r(t)))\le\mu_\kappa(\Gamma)$ give
\[
r(t)\ge \frac{R^{-1}}{\mu_\kappa(\Gamma)},
\]
and for $1\le R<R_*$ we use the fixed convention from the preceding subsection.

Order the centers,
\[
t_1<\cdots<t_{n_{\mathrm{sel}}},
\qquad
r_j=r(t_j).
\]
Disjointness of the interiors gives, for consecutive centers,
\[
t_{j+1}-t_j\ge c(r_j+r_{j+1}).
\]
We prove the reverse bound. Let $x=(t_j+t_{j+1})/2$. Since the intervals $5B_k$ cover $J$, there exists $k$ such that
\[
|x-t_k|\le5c\,r_k.
\]
There are no selected centers strictly between $t_j$ and $t_{j+1}$, so $t_k\le t_j$ or $t_k\ge t_{j+1}$. Suppose the former. Then
\[
\frac{t_{j+1}-t_j}{2}\le |x-t_k|\le5c\,r_k.
\]
Choosing $c$ so that $5c\le1$, we also have $|t_j-t_k|\le5c\,r_k\le r_k$; Lemma~\ref{lem:slow-variation} gives $r_j\asymp_{\Gamma,D}r_k$. Hence
\[
t_{j+1}-t_j\lesssim_{\Gamma,D}c\,r_j.
\]
If $t_k\ge t_{j+1}$, the same bound follows symmetrically with $r_{j+1}$ in place of $r_j$. From the outset, choose $c$ sufficiently small in terms of the slow-variation constant for $A=1$ so that either of these bounds places the two centers at distance smaller than the corresponding terminal scale. A further application of Lemma~\ref{lem:slow-variation} then gives
\[
r_{j+1}\asymp_{\Gamma,D}r_j.
\]
Substituting this comparability into the upper bound and using disjointness of $B_j$ and $B_{j+1}$ yields
\[
t_{j+1}-t_j\asymp_{\Gamma,D}r_j.
\]
Here and below the fixed constant $c$ is absorbed into structural constants.

Now take the consecutive Voronoi cells, in the standard nearest-neighbor tessellation sense~\cite{OkabeBootsSugiharaChiuKendall2000}: the interior endpoints are the midpoints $(t_j+t_{j+1})/2$, and at the endpoints of $J$ we truncate at the chart boundary. The preceding estimates show that, for each cell $I_j$,
\[
|I_j|\asymp_{\Gamma,D}r_j.
\]
For the first and last cells, we also use that the endpoints of $J$ belong to some interval $5B_k$; the ordering of the centers and another application of slow variation give the same bound. If the selected family consists of a single interval, the covering $5B_1\supset J$ gives the upper bound for $|I_1|$, whereas the lower bound is absorbed into the structural chart parameters, using $r_1\le r_*$. In particular, no arbitrarily small terminal cell occurs.

Since every point $t\in I_j$ satisfies $|t-t_j|\lesssim_{\Gamma,D}r_j$, Lemma~\ref{lem:slow-variation} implies
\[
r(t)\asymp_{\Gamma,D}r_j,
\qquad
\rho_R(\gamma(t))\asymp_{\Gamma,D}\rho_j.
\]
It remains to compare the turning mass of the cell with that of the terminal ball. The upper bound follows from
\[
I_j\subset B_\Gamma(\gamma(t_j),C r_j)
\]
and the first part of Lemma~\ref{lem:doubling-local-comparison}, with coincident centers and a fixed ratio of radii. For the lower bound, $I_j$ contains, at least on one side of $t_j$, a subinterval of length $c_1r_j$ with structural $c_1>0$. If $y$ is its midpoint, then
\[
B_\Gamma(\gamma(y),c_2r_j)\subset I_j
\]
for some $c_2>0$, while
\[
d_\Gamma(\gamma(y),\gamma(t_j))\lesssim_{\Gamma,D}r_j.
\]
The first part of Lemma~\ref{lem:doubling-local-comparison}, applied to these centers and comparable radii, gives
\[
\mu_\kappa(I_j)
\asymp_{\Gamma,D}
\mu_\kappa(B_\Gamma(\gamma(t_j),r_j)).
\]
By \eqref{eq:terminal-general} and \eqref{eq:rho-general},
\[
\mu_\kappa(B_\Gamma(\gamma(t_j),r_j))
\asymp_\Gamma \rho_j.
\]
Since the angular range of the chart is fixed,
\[
\alpha_j\asymp_\Gamma\mu_\kappa(I_j),
\]
and together with $|I_j|\asymp_{\Gamma,D}r_j$ this gives all the comparabilities in \eqref{eq:terminal-cap-general} and the product
\[
|I_j|\alpha_j\asymp_{\Gamma,D}R^{-1}.
\]
It also follows that consecutive caps have comparable lengths and angular widths.

Finally, fix $n_{\mathrm{nb}}\ge3$. If, for example, $\ell\ge j+n_{\mathrm{nb}}+1$, then at least the two full caps $I_{j+1}$ and $I_{\ell-1}$ lie between $I_j$ and $I_\ell$. Convexity makes the Gauss map monotone on the chart, and since its angular range is less than $\theta_0<\pi/2$, Euclidean distance and angular separation are comparable. Therefore
\[
\operatorname{dist}(\mathcal N(I_j),\mathcal N(I_\ell))
\gtrsim_\Gamma
\mu_\kappa(I_{j+1})+\mu_\kappa(I_{\ell-1})
\asymp_{\Gamma,D}
\alpha_j+\alpha_\ell,
\]
where in the last step we used comparability of consecutive caps. This proves \eqref{eq:angular-separation-general}. The statement about the number of neighbors follows immediately from the index definition.
\end{proof}

\begin{definition}[Admissible terminal partition]\label{def:terminal-admissible}
Let $\mathcal P_R=\{I_j\}$ be an ordered partition of a convex chart and write
\[
\alpha_I:=\operatorname{diam}\mathcal N(I),
\qquad
\delta=R^{-1}.
\]
We say that $\mathcal P_R$ is an \emph{admissible terminal partition} if the following properties hold with constants depending only on $\Gamma$ and $D$ (and, in the enlargement property, also on the fixed factor $A$), uniformly in $R$:
\begin{enumerate}
\item for every cap $I$,
\[
|I|\alpha_I\asymp_{\Gamma,D}\delta;
\]
\item consecutive caps have comparable lengths and angular widths;
\item for a fixed constant $n_{\mathrm{nb}}$, each cap has $O(1)$ neighbors by index and, if $I\not\sim J$,
\[
\operatorname{dist}(\mathcal N(I),\mathcal N(J))\gtrsim_{\Gamma,D}\alpha_I+\alpha_J;
\]
\item for each $A\ge1$, the truncated enlarged interval
\[
I^{[A]}:=J\cap\{t:|t-t_I|\le A|I|\}
\]
meets only $O_{\Gamma,D,A}(1)$ caps $K$, all satisfying
\[
|K|\asymp_{\Gamma,D,A}|I|,
\qquad
\alpha_K\asymp_{\Gamma,D,A}\alpha_I.
\]
\end{enumerate}
\end{definition}

The partition constructed in Lemma~\ref{lem:terminal-partition-general} is an admissible terminal partition. Indeed, the first three properties are precisely \eqref{eq:terminal-cap-general} and \eqref{eq:angular-separation-general}. For the fourth, if $u\in I^{[A]}$, then $|u-t_I|\lesssim_A|I|\asymp_{\Gamma,D}r_R(\gamma(t_I))$; Lemma~\ref{lem:slow-variation} gives
\[
r_R(\gamma(u))\asymp_{\Gamma,D,A}r_R(\gamma(t_I)).
\]
If a terminal cap $K$ meets $I^{[A]}$, applying \eqref{eq:terminal-cap-general} at a point of intersection gives
\[
|K|\asymp_{\Gamma,D,A}|I|,
\qquad
\alpha_K\asymp_{\Gamma,D,A}\alpha_I.
\]
Since $I^{[A]}$ has length $O_A(|I|)$, it can meet only $O_{\Gamma,D,A}(1)$ caps.

The general theory uses the partition from Lemma~\ref{lem:terminal-partition-general}; arguments depending only on the preceding properties are stated for an admissible terminal partition. For a terminal cap $I$ in the general partition, we write
\[
\rho_I:=\rho_R(\gamma(t_I)),
\qquad
\alpha_I:=\operatorname{diam}\mathcal N(I),
\]
so that
\[
\rho_I\asymp_{\Gamma,D}\alpha_I,
\qquad
R|I|\asymp_{\Gamma,D}\rho_I^{-1}.
\]

\subsection{Transverse geometry under doubling}

Work in a graph chart $\gamma(t)=(t,\phi(t))$ with bounded slope. Define $q_\phi$ as in \eqref{eq:def-qphi}.

\begin{lemma}[Transverse relation]\label{lem:q-separation-doubling}
Under doubling of the turning measure,
\[
q_\phi(\bar t,\tilde t)\asymp_{\Gamma,D} \tilde t\,\partial_{\tilde t}q_\phi(\bar t,\tilde t).
\]
Moreover,
\[
\partial_{\tilde t}q_\phi(\bar t,\tilde t)
\asymp_\Gamma
|\mathcal N(\bar t-\tilde t)-\mathcal N(\bar t+\tilde t)|.
\]
\end{lemma}

\begin{proof}
Set
\[
A_{\bar t}(r)=\phi'(\bar t+r)-\phi'(\bar t-r),
\qquad
J_r=\gamma([\bar t-r,\bar t+r]).
\]
Then
\[
q_\phi(\bar t,\tilde t)=\int_0^{\tilde t}A_{\bar t}(r)\,\mathrm dr,
\qquad
A_{\bar t}(\tilde t)=\partial_{\tilde t}q_\phi(\bar t,\tilde t).
\]
If $|\phi'|\le L_\phi$ on the chart, the derivative of $\arctan$ is bounded above and below on $[-L_\phi,L_\phi]$; hence
\[
A_{\bar t}(r)
\asymp_\Gamma
\mu_\kappa(J_r).
\]
Moreover,
\[
|J_{\tilde t}|\le 2\sqrt{1+L_\phi^2}\,\tilde t,
\qquad
|J_{\tilde t/2}|\ge \tilde t,
\]
so $J_{\tilde t/2}\subset J_{\tilde t}$ occupies a structurally fixed fraction of the length of $J_{\tilde t}$. Since the working charts have intrinsic length at most $r_*$, the second part of Lemma~\ref{lem:doubling-local-comparison} gives
\[
\mu_\kappa(J_{\tilde t/2})
\gtrsim_{\Gamma,D}
\mu_\kappa(J_{\tilde t}).
\]
Consequently,
\[
A_{\bar t}(\tilde t/2)
\gtrsim_{\Gamma,D}
A_{\bar t}(\tilde t).
\]
Convexity makes $A_{\bar t}$ increasing, and therefore
\[
\frac{\tilde t}{2}A_{\bar t}(\tilde t/2)
\le
\int_{\tilde t/2}^{\tilde t}A_{\bar t}(r)\,\mathrm dr
\le
q_\phi(\bar t,\tilde t)
\le
\tilde t\,A_{\bar t}(\tilde t).
\]
This proves
\[
q_\phi(\bar t,\tilde t)
\asymp_{\Gamma,D}
\tilde t\,A_{\bar t}(\tilde t).
\]

Finally, if $\vartheta(t)=\arctan\phi'(t)$ is the tangent angle, bounded slope gives
\[
A_{\bar t}(\tilde t)
\asymp_\Gamma
|\vartheta(\bar t+\tilde t)-\vartheta(\bar t-\tilde t)|.
\]
Since the angular range of the chart is less than $\theta_0<\pi/2$, Euclidean distance between the normals is comparable to this angular difference. Thus
\[
A_{\bar t}(\tilde t)
\asymp_\Gamma
|\mathcal N(\bar t-\tilde t)-\mathcal N(\bar t+\tilde t)|,
\]
which is the second assertion.
\end{proof}

\begin{lemma}[Near-diagonal implies few caps]\label{lem:near-diagonal-few-caps}
Let $\mathcal P_R$ be an admissible terminal partition. For each $C_{\mathrm{diag}}\ge1$, if
\[
q_\phi(\bar t,\tilde t)\le C_{\mathrm{diag}}R^{-1},
\]
then $[\bar t-\tilde t,\bar t+\tilde t]$ crosses only $O_{\Gamma,D,C_{\mathrm{diag}}}(1)$ complete terminal caps.
\end{lemma}

\begin{proof}
If it crosses $n$ caps $I_j$, then
\[
\tilde t\gtrsim_\Gamma\sum_{j=1}^n|I_j|,
\qquad
\partial_{\tilde t}q_\phi(\bar t,\tilde t)\gtrsim_\Gamma\sum_{j=1}^n\alpha_j.
\]
By Lemma~\ref{lem:q-separation-doubling} and Cauchy--Schwarz,
\[
q_\phi(\bar t,\tilde t)
\gtrsim_{\Gamma,D}
\left(\sum_j|I_j|\right)
\left(\sum_j\alpha_j\right)
\ge
\left(\sum_j\sqrt{|I_j|\alpha_j}\right)^2
\asymp_{\Gamma,D} n^2R^{-1}.
\]
The hypothesis forces $n=O_{\Gamma,D,C_{\mathrm{diag}}}(1)$.
\end{proof}

\subsection{Rectangular localization of terminal caps}

In this subsection set $\delta=R^{-1}$. Regard each graph chart as an interior subchart of a slightly larger convex chart with the same structural slope and doubling controls. Because all dilations below have fixed factor, caps near artificial chart boundaries are treated uniformly; global endpoints of an arc are handled directly through neighborhoods of $\gamma(I)$, without extending the curve beyond its domain. The global assembly uses a finite cover by interior subcharts and a subordinate partition of unity. For $R$ in a bounded interval, the partition contains $O_{\Gamma,D}(1)$ caps and the proposition follows after enlarging the constant. Thus the proof may assume that $R$ is sufficiently large that $\delta$ is smaller than the fixed chart margin.

For a terminal cap $I$, write
\[
\Pi_I=\Pi(I,\delta),
\]
with the notation from the preliminaries of this section.

\begin{lemma}[Terminal arc and enlarged neighborhood]\label{lem:terminal-rectangle-curve}
Let $\mathcal P_R$ be an admissible terminal partition. For fixed $A\ge1$, there exists $C=C(\Gamma,D,A)$ such that, for every terminal cap $I$,
\[
\gamma(I)+B(0,A\delta)\subset C\Pi_I.
\]
Moreover, if $J$ is the arc-length interval of the chart and we define the truncated enlarged interval
\[
I^{(A)}:=J\cap\{t:|t-t_I|\le C|I|\},
\]
then $I^{(A)}$ meets only $O_{\Gamma,D,A}(1)$ terminal caps, and all of them have length and angular width comparable to those of $I$.
\end{lemma}

\begin{proof}
Let $L=|I|$. By Lemma~\ref{lem:product-linearization}, for $t\in I$,
\[
\gamma(t)-\gamma(t_I)
=
a_tT(t_I)+b_t\mathcal N(t_I),
\]
with
\[
|a_t|\lesssim_\Gamma L,
\qquad
|b_t|\lesssim_\Gamma L\alpha_I
\lesssim_{\Gamma,D}\delta,
\]
where we used the admissibility product condition. Since $\alpha_I$ is uniformly bounded by the angular range of the chart, the same condition also gives
\[
\delta\lesssim_{\Gamma,D} L.
\]
If $e\in B(0,A\delta)$, its tangential and normal components are, respectively, $O_{\Gamma,D,A}(L)$ and $O_A(\delta)$. By the definition of $\Pi_I$, it follows that
\[
\gamma(t)+e\in C\Pi_I
\]
for a constant $C=C(\Gamma,D,A)$, proving the first assertion. This inclusion is one-sided and remains valid when $I$ meets a global endpoint of $\Gamma$.

The second assertion is exactly the local enlargement property in Definition~\ref{def:terminal-admissible}, applied with the fixed factor $C=C(\Gamma,D,A)$. Truncation by an endpoint of $J$ only decreases the length of the enlarged interval and does not alter the conclusion.
\end{proof}

\section{General \texorpdfstring{$L^4$}{L4} estimates}
\label{sec:direct-general}

This section uses the geometry of the terminal scale from Section~\ref{sec:terminal-geometry} to prove a local $L^4$ estimate under the standing assumption that the turning measure is doubling with constant $D$. Bounded multiplicity of the sumsets associated with the terminal partition yields essential biorthogonality; the resulting sum is then identified with the adapted energy and used to prove the estimate on balls. The final part identifies the maximal admissible size of the cutoff within the local kernel family and establishes the radius comparisons that lead to weighted estimates.

For a terminal partition $\mathcal P_R$ of a chart, we identify each interval $I$ with the corresponding subarc $\gamma(I)$. If $f\in L^2(\Gamma,\mathrm d\sigma)$, define $f_I:=f\,\mathbf 1_{\gamma(I)}$.
Under the pullback convention fixed in Section~\ref{sec:terminal-geometry}, we also write $f_I=f\,\mathbf 1_I$.

\subsection{Essential biorthogonality and sumset multiplicity}

Following the Córdoba--Fefferman square-function mechanism and Schippa's adaptation to degenerate coverings~\cite{CordobaFefferman1978,Schippa2024}, we use the following operational formulation.

\begin{definition}[Essential biorthogonality]\label{def:essential-biorthogonality}
Let $\{\Omega_I\}$ be a finite family of frequency regions. We say that it is \emph{essentially biorthogonal} if the family of sumsets
\[
\{\Omega_I+\Omega_J\}_{I,J}
\]
has uniformly bounded multiplicity, that is, if there exists a structural constant $C$ such that
\[
\sup_{\zeta\in\mathbb R^2}
\#\{(I,J):\zeta\in\Omega_I+\Omega_J\}
\le C.
\]
For functions $H_I$ with $\operatorname{supp}\widehat H_I\subset\Omega_I$, this property implies, by Plancherel and Cauchy--Schwarz,
\[
\left\|\sum_I H_I\right\|_4^4
\lesssim_C
\sum_{I,J}\|H_IH_J\|_2^2.
\]
\end{definition}

Set $\delta=R^{-1}$ throughout this subsection. We work on the convex charts and with the admissible terminal partitions constructed in the preceding section; in particular, angular separation, the near-diagonal regime, and Lemma~\ref{lem:terminal-rectangle-curve} are available.

\begin{lemma}[Geometric multiplicity of the sumsets]\label{lem:terminal-sumset-multiplicity}
Let $\mathcal P_R$ be an admissible terminal partition on a convex chart with doubling turning measure, and let $A\ge1$ be fixed. For each ordered pair $I,J\in\mathcal P_R$, define
\[
\mathcal S_{I,J}
=
\gamma(I)+\gamma(J)+B(0,2A\delta).
\]
Then the family $\{\mathcal S_{I,J}\}_{I,J}$ has multiplicity
\[
O_{\Gamma,D,A}(1).
\]
More precisely, the proof shows that if two terminal pairs produce sums at distance $O_A(\delta)$, then the corresponding caps are at combinatorial distance $O_{\Gamma,D,A}(1)$, up to interchange of the two summands.
\end{lemma}

\begin{proof}
Fix $\zeta\in\mathbb R^2$ and suppose that $\zeta\in\mathcal S_{I,J}$. There exist $u\in I$, $v\in J$ such that
\begin{equation}\label{eq:approx-sum-general}
|\zeta-\gamma(u)-\gamma(v)|\lesssim_A\delta.
\end{equation}
Order the points of each pair so that $u\le v$; interchanging the two endpoints only introduces a constant factor.

Pass to graph coordinates $\gamma(t)=(t,\phi(t))$. Since the slope is bounded, parameter length and arc length are comparable, so we continue to write $|K|$ for either. Given two pairs $(u,v)$ and $(u',v')$ satisfying \eqref{eq:approx-sum-general}, define
\[
\bar t=\frac{u+v}{2},\quad \tilde t=\frac{v-u}{2},
\qquad
\bar t'=\frac{u'+v'}2,\quad \tilde t'=\frac{v'-u'}2.
\]
Interchanging the two pairs if necessary, assume from now on that
\[
\tilde t\ge \tilde t'.
\]
Then
\[
[\bar t-\tilde t',\bar t+\tilde t']\subset[\bar t-\tilde t,\bar t+\tilde t]=[u,v],
\]
so $q_\phi(\bar t,\tilde t')$ is defined even when the first pair meets a global endpoint of the chart. Moreover, if the parameter domain of the chart is $[a,b]$, the admissible centers for the fixed radius $\tilde t'$ form the interval $[a+\tilde t',b-\tilde t']$. Both $\bar t$ and $\bar t'$ belong to this interval; by convexity, the entire segment between the two centers is admissible. In particular, every Lipschitz comparison in the $\bar t$ variable with radius $\tilde t'$ remains inside the domain of $q_\phi$.

The first coordinate of the sums gives
\begin{equation}\label{eq:center-close-general}
|\bar t-\bar t'|\lesssim_{\Gamma,A}\delta.
\end{equation}
By \eqref{eq:sum-map-center-distance}, the second coordinate and the bounded slope imply
\[
|q_\phi(\bar t,\tilde t)-q_\phi(\bar t',\tilde t')|\lesssim_{\Gamma,A}\delta.
\]
Moreover, $q_\phi$ is uniformly Lipschitz in the $\bar t$ variable on the chart because $\phi'$ is bounded. Consequently,
\begin{equation}\label{eq:q-close-general}
|q_\phi(\bar t,\tilde t)-q_\phi(\bar t,\tilde t')|\lesssim_{\Gamma,A}\delta.
\end{equation}

Fix a constant $C_{\mathrm{tr}}=C_{\mathrm{tr}}(\Gamma,D,A)$ sufficiently large, larger than the implicit constants in \eqref{eq:center-close-general} and \eqref{eq:q-close-general}, and distinguish two regimes. If
\[
\min\{q_\phi(\bar t,\tilde t),q_\phi(\bar t,\tilde t')\}\le C_{\mathrm{tr}}\delta,
\]
then both defects are $O_{\Gamma,D,A}(\delta)$ by \eqref{eq:q-close-general}; the Lipschitz estimate in $\bar t$, together with the preceding comparison between the second coordinates of the sums, also gives $q_\phi(\bar t',\tilde t')=O_{\Gamma,D,A}(\delta)$. Lemma~\ref{lem:near-diagonal-few-caps} shows that each interval $[u,v]$ and $[u',v']$ crosses only $O_{\Gamma,D,A}(1)$ caps. Let $I_0$ and $I_0'$ be the caps containing $\bar t$ and $\bar t'$, respectively. For every terminal cap $K_0$,
\[
\delta\asymp_{\Gamma,D} |K_0|\alpha_{K_0}\lesssim_{\Gamma,D} |K_0|,
\]
so \eqref{eq:center-close-general} gives $|\bar t-\bar t'|\lesssim_{\Gamma,D,A}|I_0|$. The second assertion of Lemma~\ref{lem:terminal-rectangle-curve}, applied with a sufficiently large fixed factor, implies that $I_0$ and $I_0'$ differ by only $O_{\Gamma,D,A}(1)$ indices. Since the endpoints of each pair lie within $O_{\Gamma,D,A}(1)$ caps of its center cap, all four endpoints belong to the same cluster of $O_{\Gamma,D,A}(1)$ caps. Thus, in the near-diagonal regime there are only $O_{\Gamma,D,A}(1)$ possibilities for $(I,J)$.

Now suppose that both transverse defects are larger than $C_{\mathrm{tr}}\delta$. Since $\tilde t\ge \tilde t'$ was fixed before comparing the defects, convexity of $\phi$ implies that $\partial_{\tilde t}q_\phi(\bar t,r)$ is nondecreasing in $r$, and \eqref{eq:q-close-general} gives
\begin{equation}\label{eq:separation-close-transverse-general}
(\tilde t-\tilde t')\,\partial_{\tilde t}q_\phi(\bar t,\tilde t')
\lesssim_{\Gamma,A}\delta.
\end{equation}
Before comparing the endpoints of the primed pair, we must justify moving the center from $\bar t'$ to $\bar t$. Bounded slope gives
\[
q_\phi(\bar t,\tilde t')
\le \tilde t'\,\partial_{\tilde t}q_\phi(\bar t,\tilde t')
\lesssim_\Gamma \tilde t'.
\]
Since $q_\phi(\bar t,\tilde t')\ge C_{\mathrm{tr}}\delta$, it follows that
\[
\tilde t'\gtrsim_\Gamma C_{\mathrm{tr}}\delta.
\]
Combining this with \eqref{eq:center-close-general} and choosing $C_{\mathrm{tr}}$ sufficiently large yields
\[
|\bar t-\bar t'|\le \frac12\tilde t'.
\]
Since $\tilde t'\le\tilde t$, the interval $[\bar t-\tilde t',\bar t+\tilde t']$ is contained in $[u,v]$, and $[\bar t'-\tilde t',\bar t'+\tilde t']=[u',v']$; both therefore remain inside the chart. Denote
\[
J=\gamma([\bar t-\tilde t',\bar t+\tilde t']),
\qquad
J'=\gamma([\bar t'-\tilde t',\bar t'+\tilde t']).
\]
The bound on the displacement of the centers implies that $J\cap J'$ occupies a structurally fixed fraction of the length of each. Applying the second part of Lemma~\ref{lem:doubling-local-comparison} twice, and using that the graph parametrization has speed bounded above and below, gives
\[
\mu_\kappa(J)\asymp_{\Gamma,D}\mu_\kappa(J').
\]
The same comparison between slope difference and turning mass used in Lemma~\ref{lem:q-separation-doubling} then gives
\[
\partial_{\tilde t}q_\phi(\bar t,\tilde t')
\asymp_{\Gamma,D}
\partial_{\tilde t}q_\phi(\bar t',\tilde t')
\asymp_\Gamma
|\mathcal N(u')-\mathcal N(v')|.
\]

The pair $(u',v')$ cannot lie in neighboring caps: if $I'\sim J'$, then the lengths and angular widths of the caps between the two endpoints are comparable and
\[
\tilde t'\lesssim_{\Gamma,D}|I'|,
\qquad
\partial_{\tilde t}q_\phi(\bar t,\tilde t')
\lesssim_{\Gamma,D}\alpha_{I'},
\]
so, using $q_\phi(\bar t,\tilde t')\le \tilde t'\partial_{\tilde t}q_\phi(\bar t,\tilde t')$,
\[
q_\phi(\bar t,\tilde t')\lesssim_{\Gamma,D}|I'|\alpha_{I'}\asymp_{\Gamma,D}\delta,
\]
a contradiction if $C_{\mathrm{tr}}$ was chosen sufficiently large.

Hence $I'$ and $J'$ are not neighbors. The angular separation \eqref{eq:angular-separation-general} and Lemma~\ref{lem:q-separation-doubling} give
\[
\partial_{\tilde t}q_\phi(\bar t,\tilde t')
\gtrsim_{\Gamma,D}
\alpha_{I'}+\alpha_{J'}.
\]
Inserting this into \eqref{eq:separation-close-transverse-general} and using the product condition,
\[
|\tilde t-\tilde t'|
\lesssim_{\Gamma,D,A}
\frac{\delta}{\alpha_{I'}+\alpha_{J'}}
\lesssim_{\Gamma,D}
\min\{|I'|,|J'|\}.
\]
Moreover,
\[
\delta\lesssim_{\Gamma,D}\min\{|I'|,|J'|\},
\]
since $\alpha_{I'},\alpha_{J'}$ are uniformly bounded by the angular range of the chart. Together with \eqref{eq:center-close-general}, this gives
\[
|u-u'|+|v-v'|
\lesssim_{\Gamma,D,A}
\min\{|I'|,|J'|\}.
\]
The second assertion of Lemma~\ref{lem:terminal-rectangle-curve}, applied with a sufficiently large fixed factor, then implies that $I$ differs from $I'$ by only $O_{\Gamma,D,A}(1)$ indices, and likewise for $J$ and $J'$. Thus the transverse regime also has multiplicity $O_{\Gamma,D,A}(1)$.

The preceding conclusion shows that, for fixed $\zeta$, only $O_{\Gamma,D,A}(1)$ ordered pairs $(I,J)$ can satisfy $\zeta\in\mathcal S_{I,J}$. This is precisely the claimed multiplicity.
\end{proof}

\begin{proposition}[Essential biorthogonality]\label{prop:biortho-general}
Let $\mathcal P_R$ be an admissible terminal partition on a convex chart with doubling turning measure. If
\[
\operatorname{supp}\widehat H_I
\subset
\gamma(I)+B(0,A\delta)
\]
for a fixed factor $A$, then
\[
\left\|\sum_{I\in\mathcal P_R}H_I\right\|_4^4
\lesssim_{\Gamma,D,A}
\sum_{I,J}\|H_IH_J\|_2^2.
\]
\end{proposition}

\begin{proof}
With $\mathcal S_{I,J}$ as in Lemma~\ref{lem:terminal-sumset-multiplicity}, we have
\[
\operatorname{supp}\widehat{H_IH_J}
\subset\mathcal S_{I,J},
\]
and Lemma~\ref{lem:terminal-sumset-multiplicity} gives multiplicity $O_{\Gamma,D,A}(1)$ for the frequency supports of the products. Plancherel and Cauchy--Schwarz directly yield the claimed estimate.
\end{proof}

\subsection{General energy and the local \texorpdfstring{$L^4$}{L4} estimate}

With the global terminal angular scale defined in \eqref{eq:rho-general}, set
\begin{equation}\label{eq:energy-general}
\mathcal E_R^\Gamma(f)
=
\iint_{\Gamma^2}
\frac{|f(\xi)|^2|f(\eta)|^2}
{|\mathcal N(\xi)-\mathcal N(\eta)|+\max\{\rho_R(\xi),\rho_R(\eta)\}}
\,\mathrm d\sigma(\xi)\mathrm d\sigma(\eta).
\end{equation}

\begin{proposition}[General energy--sum equivalence]\label{prop:energy-sum-general}
Let $U$ be one of the working charts and let $\mathcal P_R$ be a terminal partition of $U$. If $f\in L^2(\Gamma,\mathrm d\sigma)$ is supported in $U$, then
\begin{align*}
\mathcal E_R^\Gamma(f)
\asymp_{\Gamma,D}{}&
\sum_I\frac{\|f_I\|_2^4}{\rho_I}\\
&+\sum_{I\not\sim J}
\iint_{I\times J}
\frac{|f(\xi)|^2|f(\eta)|^2}{|\mathcal N(\xi)-\mathcal N(\eta)|}
\,\mathrm d\sigma(\xi)\mathrm d\sigma(\eta).
\end{align*}
\end{proposition}

\begin{proof}
Decompose the integral in \eqref{eq:energy-general} according to pairs of caps. If $I\sim J$, slow variation gives $\rho_R\asymp_{\Gamma,D}\rho_I\asymp_{\Gamma,D}\rho_J$ on $I\cup J$, and index adjacency gives $|\mathcal N(\xi)-\mathcal N(\eta)|\lesssim_{\Gamma,D}\rho_I$ on $I\times J$. Therefore
\[
\iint_{I\times J}
\frac{|f(\xi)|^2|f(\eta)|^2}
{|\mathcal N(\xi)-\mathcal N(\eta)|+\max\{\rho_R(\xi),\rho_R(\eta)\}}
\,\mathrm d\sigma(\xi)\mathrm d\sigma(\eta)
\asymp_{\Gamma,D}
\frac{\|f_I\|_2^2\|f_J\|_2^2}{\rho_I}.
\]
Each cap has $O_{\Gamma,D}(1)$ neighbors and neighboring cap scales are comparable; an application of $ab\le(a^2+b^2)/2$ reduces the sum of these terms to the first sum in the statement, while the terms $I=J$ give the reverse bound.

If $I\not\sim J$, \eqref{eq:angular-separation-general} implies
\[
|\mathcal N(\xi)-\mathcal N(\eta)|
\gtrsim_{\Gamma,D}
\alpha_I+\alpha_J
\asymp_{\Gamma,D}
\rho_I+\rho_J
\]
uniformly on $I\times J$. The cutoff in the denominator is then absorbed by the angular separation, and the kernel is comparable to $|\mathcal N(\xi)-\mathcal N(\eta)|^{-1}$. Summing the two regimes gives the equivalence.
\end{proof}

Fix $x_0\in\mathbb R^2$ and a Schwartz function $\psi$ such that $\psi\ge c>0$ on the unit ball and $\widehat\psi$ is supported in a fixed ball, and set $\psi_R(x)=\psi((x-x_0)/R)$.

\begin{lemma}[General local $L^2$ estimate]\label{lem:local-L2-general}
For every terminal cap $I$,
\[
\int|\psi_R|^4|E_\Gamma f_I|^2\,\mathrm dx
\lesssim_\Gamma R\|f_I\|_2^2.
\]
\end{lemma}

\begin{proof}
Expanding the integral produces the kernel
\[
e^{2\pi i x_0\cdot(\xi-\eta)}
R^2\widehat{|\psi|^4}\bigl(-R(\xi-\eta)\bigr),
\qquad \xi,\eta\in\Gamma.
\]
The phase has modulus one. On a graph chart with bounded slope there exists $c_\Gamma>0$ such that
\[
|\gamma(t)-\gamma(t')|\ge c_\Gamma|t-t'|,
\]
and arc-length measure is comparable to $\mathrm dt$. Schwartz decay of $\widehat{|\psi|^4}$ gives, uniformly in a point of the chart,
\[
\int R^2\bigl|\widehat{|\psi|^4}(R(\gamma(t)-\gamma(t')))\bigr|\,\mathrm dt'
\lesssim_\Gamma R.
\]
Schur's test (cf. Stein~\cite{Stein1993}) gives the bound.
\end{proof}

\begin{theorem}[Local $L^4$ estimate on balls]\label{thm:direct-general-ball}
Let $\Gamma$ be a compact convex $C^2$ curve with no affine subarcs and whose turning measure is doubling with constant $D$. Then, for every $R\ge1$ and $x_0\in\mathbb R^2$,
\[
\int_{B(x_0,R)}|E_\Gamma f(x)|^4\,\mathrm dx
\lesssim_{\Gamma,D}
\mathcal E_R^\Gamma(f).
\]
\end{theorem}

\begin{proof}
First suppose that $f$ is supported in one of the working charts, and let $\mathcal P_R$ be its terminal partition. Apply Proposition~\ref{prop:biortho-general} to $H_I=\psi_RE_\Gamma f_I$.
Since $\widehat\psi_R$ is supported in a ball of radius $O(R^{-1})$,
\[
\operatorname{supp}\widehat H_I
\subset
\gamma(I)+B(0,A R^{-1})
\]
for a fixed factor $A=A(\psi)$. This is exactly the hypothesis of Proposition~\ref{prop:biortho-general}; Lemma~\ref{lem:terminal-rectangle-curve} also shows that this frequency neighborhood is contained in a fixed dilation of the terminal rectangle. Neighboring pairs are controlled by combining Lemma~\ref{lem:local-L2-general} with
\[
\|E_\Gamma f_J\|_\infty^2
\le \sigma(J)\|f_J\|_2^2
\lesssim_\Gamma |J|\|f_J\|_2^2
\]
and the identity $R|J|\asymp_{\Gamma,D}\rho_J^{-1}$; the $O_{\Gamma,D}(1)$ degree of the adjacency relation allows these terms to be summed. For nonneighboring pairs, order the two subarcs inside the chart and use the exact bilinear identity from Proposition~\ref{prop:bilinear-identity}; convexity makes the sum map injective on the ordered region, and \eqref{eq:angular-separation-general} converts the determinant into normal separation. Proposition~\ref{prop:energy-sum-general} assembles the two regimes and gives the estimate on one chart.

To pass to the whole curve, take a partition of unity $f=\sum_af_a$ subordinate to a finite chart cover with bounded overlap. Then
\[
|E_\Gamma f|^4\lesssim_\Gamma\sum_a|E_\Gamma f_a|^4
\]
and, by positivity of the energy kernel and finite overlap,
\[
\sum_a\mathcal E_R^\Gamma(f_a)\lesssim_\Gamma\mathcal E_R^\Gamma(f).
\]
Since the number of charts is structural, these two estimates assemble the global bound.
\end{proof}

\begin{corollary}[Logarithmic comparison for nondegenerate curvature]\label{cor:nondegenerate-logarithmic-bound}
Under the hypotheses of Theorem~\ref{thm:direct-general-ball}, assume in addition that the curvature of $\Gamma$ satisfies
\[
0<c_\kappa\le \kappa(\xi)\le C_\kappa<\infty,
\qquad \xi\in\Gamma.
\]
Then
\[
\mathcal E_R^\Gamma(f)
\lesssim_{\Gamma,c_\kappa,C_\kappa}
(1+\log R)\|f\|_4^4,
\]
and consequently
\[
\|E_\Gamma f\|_{L^4(B(x_0,R))}
\lesssim_{\Gamma,c_\kappa,C_\kappa}
(1+\log R)^{1/4}\|f\|_4.
\]
More generally, if $\mathcal N(\operatorname{supp}f)$ is contained in an arc of $\mathbb S^1$ of angular length $\delta\le1$, then
\[
\mathcal E_R^\Gamma(f)
\lesssim_{\Gamma,c_\kappa,C_\kappa}
\bigl(1+\log(1+\delta R^{1/2})\bigr)\|f\|_4^4.
\]
Consequently, under the same support hypothesis,
\[
\|E_\Gamma f\|_{L^4(B(x_0,R))}
\lesssim_{\Gamma,c_\kappa,C_\kappa}
\bigl(1+\log(1+\delta R^{1/2})\bigr)^{1/4}\|f\|_4.
\]
In particular, when $\delta\ge R^{-1/2}$, the logarithmic factor depends on the angular length effectively occupied by the data.
\end{corollary}

\begin{proof}
The two-sided curvature bounds imply
\[
\mu_\kappa(B_\Gamma(\xi,r))\asymp_{\Gamma,c_\kappa,C_\kappa} r
\]
in the structural range, and hence the definition of the terminal scale gives
\[
r_R(\xi)\asymp_{\Gamma,c_\kappa,C_\kappa}R^{-1/2},
\qquad
\rho_R(\xi)\asymp_{\Gamma,c_\kappa,C_\kappa}R^{-1/2}.
\]
Parametrize the normal image by the angle $\vartheta$. Since
\[
\mathrm d\vartheta=\kappa\,\mathrm d\sigma,
\]
the measure $\nu_f$ has density
\[
H(\vartheta)=\frac{|f(\mathcal N^{-1}(e^{i\vartheta}))|^2}{\kappa(\mathcal N^{-1}(e^{i\vartheta}))}
\]
with respect to $\mathrm d\vartheta$, and
\[
\|H\|_2^2
=
\int_\Gamma \frac{|f|^4}{\kappa}\,\mathrm d\sigma
\lesssim_{c_\kappa}\|f\|_4^4.
\]
By comparability of angular distance and chordal distance on $\mathbb S^1$, the energy is bounded, up to structural constants, by the convolution form associated with
\[
K_R(\vartheta)
=
\frac1{|e^{i\vartheta}-1|+R^{-1/2}}.
\]
Moreover,
\[
\|K_R\|_{L^1(\mathbb S^1)}\lesssim 1+\log R.
\]
Young's inequality therefore gives
\[
\mathcal E_R^\Gamma(f)
\lesssim
\|K_R\|_1\|H\|_2^2
\lesssim_{\Gamma,c_\kappa,C_\kappa}
(1+\log R)\|f\|_4^4.
\]
The estimate for $E_\Gamma f$ follows from Theorem~\ref{thm:direct-general-ball}.

If $H$ is supported on an angular arc of length $\delta\le1$, only angular differences of size $O(\delta)$ occur. The same argument therefore uses the truncated kernel, for which
\[
\int_{|\vartheta|\lesssim\delta}
\frac{\mathrm d\vartheta}{|\vartheta|+R^{-1/2}}
\lesssim
1+\log(1+\delta R^{1/2}),
\]
and proves the localized version.
\end{proof}

\subsection{Maximal size of the angular cutoff}

Let $I$ be a terminal cap of length $L$, angular width $\alpha_I$, and mass $\mu_I=\|f_I\|_2^2$. The terminal identity gives
\[
L\alpha_I\asymp_{\Gamma,D} R^{-1}.
\]
Within a local Riesz-type family, the size $\alpha_I$ is, up to structural constants, the largest cutoff compatible with the diagonal cost.

\begin{proposition}[Maximal size of the terminal cutoff]\label{prop:maximal-cutoff-general}
Let
\[
K_{\widehat\rho}(\omega,\omega')
=
\frac1{|\omega-\omega'|+\widehat\rho_I},
\qquad \widehat\rho_I>0,
\]
be a local kernel with constant cutoff on $\mathcal N(I)\times\mathcal N(I)$. Assume there exist constants $c,C>0$, independent of $R$, $I$, and $f_I$ but allowed to depend on $\Gamma$ and $D$, such that, for every terminal cap $I$ and every datum $f_I$ supported in $I$,
\[
\int_{B(0,cR)}|E_\Gamma f_I(x)|^4\,\mathrm dx
\le
C
\iint_{\mathcal N(I)^2}K_{\widehat\rho}(\omega,\omega')\,
\mathrm d\nu_I(\omega)\mathrm d\nu_I(\omega'),
\]
where $\nu_I=\mathcal N_\#(|f_I|^2\,\mathrm d\sigma)$. Then
\[
\widehat\rho_I\lesssim_{\Gamma,D} (RL)^{-1}\asymp_{\Gamma,D}\alpha_I.
\]
In particular, the canonical cutoff $\rho_I\asymp_{\Gamma,D}\alpha_I$ is maximal up to structural constants among these cutoffs. The upper bound determines only this maximal size.
\end{proposition}

\begin{proof}
The relevant diagonal contribution has size $RL\mu_I^2$. The following constant-amplitude test attains this scale. Indeed, parametrizing $I$ by arc length, take
\[
f_I=(\mu_I/L)^{1/2}\mathbf 1_I.
\]
By Lemma~\ref{lem:product-linearization}, after factoring out the constant phase associated with a reference point of the cap, the remaining phase varies by a uniformly small amount on a fixed fraction of the dual rectangle of dimensions $L^{-1}\times R$. On that region,
\[
|E_\Gamma f_I|\gtrsim_{\Gamma,D} (L\mu_I)^{1/2}.
\]
Moreover, $L^{-1}\lesssim_{\Gamma,D} R$, because $L\alpha_I\asymp_{\Gamma,D} R^{-1}$ and the angular range of the chart is uniformly bounded. Choosing the constants of the dual rectangle sufficiently small, it lies inside $B(0,cR)$ and has area comparable to $R/L$. Therefore
\[
\int_{B(0,cR)}|E_\Gamma f_I|^4\,\mathrm dx
\gtrsim_{\Gamma,D}
RL\mu_I^2.
\]

On the other hand, for any datum of mass $\mu_I$,
\[
\iint_{\mathcal N(I)^2}K_{\widehat\rho}(\omega,\omega')\,
\mathrm d\nu_I(\omega)\mathrm d\nu_I(\omega')
\le
\frac{\mu_I^2}{\widehat\rho_I}.
\]
If this energy is to dominate the preceding test uniformly, necessarily
\[
RL\mu_I^2
\lesssim_{\Gamma,D}
\frac{\mu_I^2}{\widehat\rho_I},
\]
and hence
\[
\widehat\rho_I\lesssim_{\Gamma,D}(RL)^{-1}\asymp_{\Gamma,D}\alpha_I.
\]

The maximality is one-sided: decreasing the competing cutoff increases the kernel pointwise. Therefore any upper bound valid with $\widehat\rho_I$ remains valid with a smaller cutoff. Unique selection of the kernel requires an additional principle.
\end{proof}

\subsection{Radius comparison and weights}

The doubling hypothesis already contains a quantitative form of lower growth.

\begin{proposition}[Radius comparison from doubling]\label{prop:radius-comparison-doubling}
Define $Q_D=\log_2D$ and $\beta_D=\frac{Q_D}{Q_D+1}$. Then, for $R\ge1$, $\lambda\ge1$, and $\xi\in\Gamma$,
\[
\rho_{\lambda R}(\xi)
\gtrsim_{\Gamma,D}
\lambda^{-\beta_D}\rho_R(\xi),
\]
and
\[
\mathcal E_{\lambda R}^\Gamma(f)
\lesssim_{\Gamma,D}
\lambda^{\beta_D}\mathcal E_R^\Gamma(f).
\]
\end{proposition}

\begin{proof}
First suppose $R\ge R_*(\Gamma)$ and write
\[
\ell=r_R(\xi),
\qquad
\ell'=r_{\lambda R}(\xi).
\]
Monotonicity of $r\mapsto r\mu_\kappa(B_\Gamma(\xi,r))$ gives $\ell'\le\ell$, and by the exact definition of the terminal scale in this range,
\[
\frac{\ell'\mu_\kappa(B_\Gamma(\xi,\ell'))}
{\ell\mu_\kappa(B_\Gamma(\xi,\ell))}
=
\frac1\lambda.
\]
Lemma~\ref{lem:doubling-lower-growth} implies
\[
\frac1\lambda
\ge
D^{-1}\left(\frac{\ell'}\ell\right)^{Q_D+1}.
\]
Therefore
\[
\frac{\ell'}\ell
\lesssim_D
\lambda^{-1/(Q_D+1)}.
\]
Since $\rho_R=R^{-1}/r_R$,
\[
\rho_{\lambda R}(\xi)
\gtrsim_D
\lambda^{-Q_D/(Q_D+1)}\rho_R(\xi).
\]

For $1\le R<R_*$, use the convention $r_R=r_{R_*}$. If $\lambda R<R_*$, then $\rho_{\lambda R}=\lambda^{-1}\rho_R$, and the preceding comparison follows after adjusting a constant depending on $R_*$. If $\lambda R\ge R_*$, first compare $\lambda R$ with $R_*$ and then use
\[
\rho_{R_*}=\frac R{R_*}\rho_R.
\]
This proves the comparison uniformly for all $R\ge1$.

Finally, if
\[
a=|\mathcal N(\xi)-\mathcal N(\eta)|,
\qquad
b_R=\max\{\rho_R(\xi),\rho_R(\eta)\},
\]
then
\[
a+b_{\lambda R}
\gtrsim_{\Gamma,D}
\lambda^{-\beta_D}(a+b_R),
\]
and the energy comparison follows by integrating the corresponding kernels.
\end{proof}

\begin{corollary}[Weighted estimate under doubling]\label{cor:weighted-doubling}
If $\iota>\beta_D=\frac{\log_2D}{1+\log_2D}$, then, for every $x_0\in\mathbb R^2$,
\[
\int_{\mathbb R^2}|E_\Gamma f(x)|^4
\left(1+\frac{|x-x_0|}{R}\right)^{-\iota}\,\mathrm dx
\lesssim_{\Gamma,D,\iota}
\mathcal E_R^\Gamma(f).
\]
\end{corollary}

\begin{proof}
Decompose $\mathbb R^2$ into annuli of radii $2^jR$. Theorem~\ref{thm:direct-general-ball} and Proposition~\ref{prop:radius-comparison-doubling} reduce the sum to
\[
\sum_{j\ge0}2^{-j\iota}2^{j\beta_D},
\]
which converges under the stated hypothesis.
\end{proof}

The exponent $\beta_D$ uses only the doubling constant. Quantitative finite-type data yield a uniform exponent from the order bound $m_*$. For a fixed finite-type curve, these data may be chosen with $m_*=m_{\max}(\Gamma)$, recovering the geometric exponent determined by the largest order actually present.

\begin{proposition}[Finite-type radius comparison]\label{prop:radius-comparison-general}
Assume that $\Gamma$ has quantitative finite-type data $\mathfrak T$, and let $\beta_*=\frac{m_*+1}{m_*+2}$. Then, for $R\ge1$, $\lambda\ge1$, and $\xi\in\Gamma$,
\[
\rho_{\lambda R}(\xi)
\gtrsim_{\Gamma,\mathfrak T}
\lambda^{-\beta_*}\rho_R(\xi),
\]
and
\[
\mathcal E_{\lambda R}^\Gamma(f)
\lesssim_{\Gamma,\mathfrak T}
\lambda^{\beta_*}\mathcal E_R^\Gamma(f).
\]
\end{proposition}

\begin{proof}
Let $r_{\mathrm{lg}}$ be the radius from Lemma~\ref{lem:finite-type-lower-growth}, and let $R_{\mathrm{ft}}=R_{\mathrm{ft}}(\Gamma,\mathfrak T)$ be the single threshold fixed in Section~\ref{sec:terminal-geometry}, so that
\[
r_R(\xi)\le\min\{r_{\mathrm{ft}},r_{\mathrm{lg}}\}
\qquad
(R\ge R_{\mathrm{ft}},\ \xi\in\Gamma).
\]
First suppose $R\ge R_{\mathrm{ft}}$ and write
\[
\ell=r_R(\xi),
\qquad
\ell'=r_{\lambda R}(\xi).
\]
The exact definition of the terminal scale gives $\ell'\le\ell\le r_{\mathrm{lg}}$ and
\[
\frac{\ell'\mu_\kappa(B_\Gamma(\xi,\ell'))}
{\ell\mu_\kappa(B_\Gamma(\xi,\ell))}
=\frac1\lambda.
\]
Applying \eqref{eq:lower-growth},
\[
\lambda^{-1}
\gtrsim_{\mathfrak T}
\left(\frac{\ell'}\ell\right)^{m_*+2},
\]
so
\[
\frac{\ell'}\ell
\lesssim_{\mathfrak T}
\lambda^{-1/(m_*+2)}.
\]
Using $\rho_R=R^{-1}/\ell$ gives
\[
\rho_{\lambda R}(\xi)
\gtrsim_{\mathfrak T}
\lambda^{-(m_*+1)/(m_*+2)}\rho_R(\xi)
=
\lambda^{-\beta_*}\rho_R(\xi).
\]

It remains to treat the bounded range $1\le R<R_{\mathrm{ft}}$. In this range there are structural constants $0<c<C<\infty$ such that
\[
c\le \rho_R(\xi)\le C,
\qquad
\xi\in\Gamma,
\]
because $r_R(\xi)\le r_*$ and, for radii defined by the terminal equation, $r_R(\xi)\ge (R_{\mathrm{ft}}\mu_\kappa(\Gamma))^{-1}$; the range $R<R_*(\Gamma)$ is covered by the freezing convention. If $\lambda R<R_{\mathrm{ft}}$, this uniform comparability gives
\[
\rho_{\lambda R}(\xi)\gtrsim_{\Gamma,\mathfrak T}\rho_R(\xi)
\ge
\lambda^{-\beta_*}\rho_R(\xi).
\]
If $\lambda R\ge R_{\mathrm{ft}}$, apply the already proved case to the quotient $(\lambda R)/R_{\mathrm{ft}}$ and use the uniform comparability between $\rho_R$ and $\rho_{R_{\mathrm{ft}}}$; since $R\le R_{\mathrm{ft}}$, we again obtain
\[
\rho_{\lambda R}(\xi)
\gtrsim_{\Gamma,\mathfrak T}
\lambda^{-\beta_*}\rho_R(\xi).
\]
This proves the estimate for all $R\ge1$.

The energy comparison follows from the same kernel argument used at the end of the proof of Proposition~\ref{prop:radius-comparison-doubling}, replacing $\beta_D$ by $\beta_*$ and doubling dependencies by dependence on $\mathfrak T$.
\end{proof}

\begin{corollary}[Weighted estimate with finite-type exponent]\label{cor:weighted-general}
Under the quantitative finite-type data of Proposition~\ref{prop:radius-comparison-general}, if $\iota>\beta_*=\frac{m_*+1}{m_*+2}$, then, for every $x_0\in\mathbb R^2$,
\[
\int_{\mathbb R^2}|E_\Gamma f(x)|^4
\left(1+\frac{|x-x_0|}{R}\right)^{-\iota}\,\mathrm dx
\lesssim_{\Gamma,\mathfrak T,\iota}
\mathcal E_R^\Gamma(f).
\]
\end{corollary}

\begin{proof}
Repeat the annular decomposition from Corollary~\ref{cor:weighted-doubling}, using Proposition~\ref{prop:radius-comparison-general}.
\end{proof}

\begin{proposition}[Subendpoint necessity of the weighted threshold]\label{prop:weighted-threshold-necessary}
Assume that, on one side of $t_0$, the curvature satisfies
\[
\kappa(t_0+u)\asymp_\Gamma u^m,
\qquad 0\le u\le u_0,
\qquad m\ge0,
\]
so that $m=0$ corresponds to the nondegenerate case and $m\ge1$ to a flat point of order $m$. Write $\beta_m=\frac{m+1}{m+2}$. Fix a structural radius $R_0\ge R_*(\Gamma)$. If $\iota<\beta_m$, there is no uniform constant $C$ such that
\[
\int_{\mathbb R^2}|E_\Gamma f(x)|^4
\left(1+\frac{|x|}{R_0}\right)^{-\iota}\,\mathrm dx
\le C\,\mathcal E_{R_0}^\Gamma(f)
\]
for every $f$ supported near $t_0$.

In particular, for a fixed finite-type curve define
\[
m_{\max}=m_{\max}(\Gamma),
\qquad
\beta_{\max}=\frac{m_{\max}+1}{m_{\max}+2}.
\]
If $m_{\max}>0$, apply the argument at a flat point where this order is attained; if $m_{\max}=0$, apply it at any nondegenerate point. In both cases, every uniform weighted inequality of this form for the fixed curve requires
\[
\iota\ge\beta_{\max}.
\]
\end{proposition}

\begin{proof}
Parametrize one side of $t_0$ by arc length as $u\mapsto\gamma(t_0+u)$, $0\le u\le u_0$, and set $T_0=T(t_0)$, $\mathcal N_0=\mathcal N(t_0)$. If $\vartheta(u)$ is the angle between $T(t_0+u)$ and $T_0$, then, by the vanishing order,
\[
|\vartheta(u)|\lesssim_\Gamma u^{m+1}.
\]
Writing
\[
\gamma(t_0+u)-\gamma(t_0)=a(u)T_0+b(u)\mathcal N_0,
\]
the arc-length parametrization gives $|a(u)|\le u$, whereas
\[
|b(u)|\le\int_0^u|\sin\vartheta(v)|\,\mathrm dv
\lesssim_\Gamma u^{m+2}.
\]
Set $k=m+2$ and, for $0<\varepsilon\ll1$,
\[
f_\varepsilon(\gamma(t_0+u))
=\varepsilon^{-1/2}\mathbf 1_{[0,\varepsilon]}(u).
\]
Then $\|f_\varepsilon\|_2=1$.

Choose sufficiently small constants $c_1,c_2>0$ and define
\[
\mathcal R_\varepsilon
=
\left\{
 x=yT_0+z\mathcal N_0:
 |y|\le c_1\varepsilon^{-1},\quad
 c_2\varepsilon^{-k}\le z\le2c_2\varepsilon^{-k}
\right\}.
\]
For $x\in\mathcal R_\varepsilon$ and $0\le u\le\varepsilon$,
\[
\bigl|x\cdot(\gamma(t_0+u)-\gamma(t_0))\bigr|
\le c_1+C_\Gamma c_2.
\]
Taking $c_1,c_2$ sufficiently small, the residual phase remains in a fixed arc on which its real part is positive. After factoring out the constant phase $e^{2\pi i x\cdot\gamma(t_0)}$, one obtains
\[
|E_\Gamma f_\varepsilon(x)|\gtrsim_\Gamma\varepsilon^{1/2}
\qquad(x\in\mathcal R_\varepsilon).
\]
Moreover,
\[
|\mathcal R_\varepsilon|\asymp_\Gamma\varepsilon^{-(k+1)},
\qquad
|x|\asymp_\Gamma\varepsilon^{-k}
\quad(x\in\mathcal R_\varepsilon),
\]
so
\[
\int_{\mathbb R^2}|E_\Gamma f_\varepsilon(x)|^4
\left(1+\frac{|x|}{R_0}\right)^{-\iota}\,\mathrm dx
\gtrsim_{\Gamma,R_0,\iota}
\varepsilon^{k\iota-(k-1)}.
\]

On the other hand, at the fixed radius $R_0$ the terminal angular scale has a positive minimum. Indeed,
\[
\rho_{R_0}(t)=\frac{R_0^{-1}}{r_{R_0}(t)}
\ge\frac{R_0^{-1}}{r_*}.
\]
Since $\|f_\varepsilon\|_2=1$,
\[
\mathcal E_{R_0}^\Gamma(f_\varepsilon)
\le\frac1{\inf_t\rho_{R_0}(t)}
\lesssim_{\Gamma,R_0}1.
\]
If $\iota<(k-1)/k$, then $k\iota-(k-1)<0$, and the left-hand side diverges as $\varepsilon\downarrow0$, whereas the energy remains uniformly bounded. This rules out the uniform inequality.
\end{proof}

For a fixed finite-type curve, choose quantitative data with $m_*=m_{\max}(\Gamma)$. Then Corollary~\ref{cor:weighted-general} gives sufficiency for $\iota>\beta_{\max}=\frac{m_{\max}(\Gamma)+1}{m_{\max}(\Gamma)+2}$, whereas Proposition~\ref{prop:weighted-threshold-necessary} gives failure for $\iota<\beta_{\max}$. Thus, for a fixed finite-type curve, the only undecided case is the endpoint $\iota=\beta_{\max}$.

For a family described by uniform data whose order bound is $m_*$, Corollary~\ref{cor:weighted-general} gives uniform sufficiency for $\iota>\beta_*=\frac{m_*+1}{m_*+2}$. This value is necessary for the class only if some curve in the family attains order $m_*$; a redundant bound $m_*>m_{\max}(\Gamma)$ for a particular curve creates no additional obstruction.

\section{The monomial model}
\label{sec:monomial}

We specialize the preceding theory to the model
\[
\phi_k(t)=t^k,
\qquad
\gamma_k(t)=(t,\phi_k(t)),
\qquad
0\le t\le1,
\qquad
k\ge3\ \text{real},
\]
and write $\Gamma_k=\gamma_k([0,1])$. Even when $k\notin\mathbb N$, the parametrization is $C^2$ on $[0,1]$, regular, convex, and contains no affine subarcs. Its turning measure is doubling, so Theorem~\ref{thm:direct-general-ball} applies throughout the full range of real $k\ge3$. A further advantage of the model is that both the terminal scale and the angular cutoff admit explicit formulas.

\subsection{Normals and transverse geometry}

The unit normal is
\begin{equation}\label{eq:normal-monomial}
\mathcal N_k(t)=\frac{(-kt^{k-1},1)}{\sqrt{1+k^2t^{2k-2}}}.
\end{equation}

The Gauss-map parametrization allows us to compare angular distance quantitatively with the variable $t^{k-1}$.

\begin{lemma}[Normal--parameter comparison]\label{lem:normal-param-monomial}
For $u,v\in[0,1]$,
\[
\frac{2k}{\pi(1+k^2)}|u^{k-1}-v^{k-1}|
\le
|\mathcal N_k(u)-\mathcal N_k(v)|
\le
k|u^{k-1}-v^{k-1}|.
\]
\end{lemma}

\begin{proof}
Write
\[
\mathcal N_k(t)=(-\sin\vartheta(t),\cos\vartheta(t)),
\qquad
\vartheta(t)=\arctan(kt^{k-1}).
\]
If $a,b\in[0,k]$, the mean value theorem gives
\[
\frac{|a-b|}{1+k^2}
\le
|\arctan a-\arctan b|
\le
|a-b|.
\]
Moreover, for $|\vartheta_1-\vartheta_2|\le\pi/2$,
\[
\frac2\pi|\vartheta_1-\vartheta_2|
\le
|e^{i\vartheta_1}-e^{i\vartheta_2}|
\le
|\vartheta_1-\vartheta_2|.
\]
Since $\vartheta([0,1])\subset[0,\arctan k]\subset[0,\pi/2)$, it is enough to combine both estimates with $a=ku^{k-1}$ and $b=kv^{k-1}$.
\end{proof}

For the transverse defect, set
\[
q_{\phi_k}(\bar t,\tilde t)
=(\bar t+\tilde t)^k+(\bar t-\tilde t)^k-2\bar t^k.
\]

\begin{lemma}[Monomial transverse geometry]\label{lem:q-monomial}
If $0\le \tilde t\le \bar t\le1$, then
\[
q_{\phi_k}(\bar t,\tilde t)\asymp_k \bar t^{k-2}\tilde t^2,
\qquad
\partial_{\tilde t}q_{\phi_k}(\bar t,\tilde t)\asymp_k \bar t^{k-2}\tilde t,
\]
and therefore
\[
q_{\phi_k}(\bar t,\tilde t)
\asymp_k
\tilde t\,\partial_{\tilde t}q_{\phi_k}(\bar t,\tilde t).
\]
\end{lemma}

\begin{proof}
For $\tilde t=0$ the assertions are immediate. If $\tilde t>0$, Taylor's formula with integral remainder gives
\[
q_{\phi_k}(\bar t,\tilde t)
=
\int_{-\tilde t}^{\tilde t}
(\tilde t-|u|)k(k-1)(\bar t+u)^{k-2}\,\mathrm du.
\]
Since $\tilde t\le\bar t$, on a fixed fraction of the integration interval one has $\bar t+u\asymp\bar t$, which gives the lower bound; the upper bound follows from $\bar t+u\le2\bar t$. Moreover,
\[
\partial_{\tilde t}q_{\phi_k}(\bar t,\tilde t)
=
k\bigl((\bar t+\tilde t)^{k-1}-(\bar t-\tilde t)^{k-1}\bigr),
\]
and the mean value theorem gives the second comparability. The last follows by combining the first two.
\end{proof}

\subsection{Terminal scales and canonical partition}

Define
\begin{equation}\label{eq:ell-monomial}
\ell_R(t)
=
R^{-1/2}(t+R^{-1/k})^{-(k-2)/2},
\end{equation}
\begin{equation}\label{eq:tau-monomial}
\tau_R(t)
=
R^{-1/2}(t+R^{-1/k})^{(k-2)/2}.
\end{equation}
Then
\begin{equation}\label{eq:ell-tau-product}
\ell_R(t)\tau_R(t)=R^{-1}.
\end{equation}
In particular,
\[
\ell_R(t)\asymp_k
\begin{cases}
R^{-1/k},&t\lesssim R^{-1/k},\\
R^{-1/2}t^{-(k-2)/2},&t\gtrsim R^{-1/k},
\end{cases}
\]
and
\[
\tau_R(t)\asymp_k
\begin{cases}
R^{-(k-1)/k},&t\lesssim R^{-1/k},\\
R^{-1/2}t^{(k-2)/2},&t\gtrsim R^{-1/k}.
\end{cases}
\]

Up to constants depending on $k$, these are exactly the intrinsic scales from Section~\ref{sec:terminal-geometry} specialized to the monomial model. Indeed, the tangent angle is $\vartheta(t)=\arctan(kt^{k-1})$, so
\[
\mathrm d\mu_\kappa(t)
=
\vartheta'(t)\,\mathrm dt
=
\frac{k(k-1)t^{k-2}}{1+k^2t^{2k-2}}\,\mathrm dt
\asymp_k
t^{k-2}\,\mathrm dt.
\]
Since $|\gamma_k'(t)|\asymp_k1$, for radii in the structural range, with the natural endpoint truncation,
\[
\mu_\kappa(B_{\Gamma_k}(\gamma_k(t),r))
\asymp_k
r(t+r)^{k-2}.
\]
In particular, $\mu_\kappa$ is doubling with constant depending only on $k$. If $a=R^{-1/k}$, then
\[
\frac{\ell_R(t)}{t+a}
=
\left(\frac{a}{t+a}\right)^{k/2}
\le1,
\]
and moreover $t+\ell_R(t)\asymp_k t+a$. Consequently,
\[
\ell_R(t)^2(t+\ell_R(t))^{k-2}\asymp_k R^{-1}.
\]
Since $r\,\mu_\kappa(B_{\Gamma_k}(\gamma_k(t),r))\asymp_k r^2(t+r)^{k-2}$, substituting $r=\ell_R(t)$ yields a quantity comparable to $R^{-1}$. Multiplying $\ell_R(t)$ by a fixed small constant reduces this product by a uniform quadratic factor, while multiplying it by a fixed large constant increases it by a uniform quadratic factor. The intrinsic definition and monotonicity of $r\mapsto r\mu_\kappa(B_{\Gamma_k}(\gamma_k(t),r))$ then give, including the bounded range of $R$ through the convention from Section~\ref{sec:terminal-geometry},
\begin{equation}\label{eq:monomial-intrinsic-scales}
r_R(\gamma_k(t))\asymp_k\ell_R(t),
\qquad
\rho_R(\gamma_k(t))\asymp_k\tau_R(t).
\end{equation}
For the integer exponents considered by Schippa, $\ell_R$ agrees in order with the maximal linearization scale of his canonical monomial covering~\cite{Schippa2024}; here the same comparison follows directly for every real $k\ge3$.

Discretize the tangential scale by
\[
\Xi(t)
=
\int_0^t\frac{\mathrm du}{\ell_R(u)}
=
\frac2k\left[\left(1+tR^{1/k}\right)^{k/2}-1\right].
\]
Since $\Xi(1)>1$, write $\Xi(1)=n+\varepsilon$, where $n\in\mathbb N$ and $0\le\varepsilon<1$,
and define $t_j$ by $\Xi(t_j)=j$ for $0\le j\le n$. If $\varepsilon=0$, take
\[
\theta_j=[t_j,t_{j+1}],
\qquad
0\le j\le n-1.
\]
If $0<\varepsilon<1$, take these intervals for $0\le j\le n-2$ and merge the final remainder with the preceding cap:
\[
\theta_{n-1}=[t_{n-1},1].
\]
Thus, for every canonical cap $\theta=[a,b]$,
\begin{equation}\label{eq:Xi-cap-length}
1\le \Xi(b)-\Xi(a)\le2.
\end{equation}
In particular, no final cap is arbitrarily smaller than the local terminal scale.

Let $\Gamma_\theta=\gamma_k(\theta)$ denote the corresponding subarc. We identify a function on $\Gamma_k$ with its pullback by $\gamma_k$ and define explicitly $f_\theta:=f\,\mathbf 1_\theta$, equivalently, $f\,\mathbf 1_{\Gamma_\theta}$ on the curve. In this section $|\theta|$ denotes parameter length; since $|\gamma_k'|\asymp_k1$, $\sigma(\Gamma_\theta)\asymp_k|\theta|$.

\begin{lemma}[Cap size and angular diameter]\label{lem:cap-gauss-diameter}
For a canonical cap $\theta$ and any representative $t_\theta\in\theta$,
\[
|\theta|\asymp_k\ell_R(t_\theta),
\qquad
\operatorname{diam}\mathcal N_k(\theta)\asymp_k\tau_R(t_\theta).
\]
Moreover, consecutive caps have comparable lengths and angular widths.
\end{lemma}

\begin{proof}
Explicit inversion of the adapted coordinate gives
\[
1+tR^{1/k}
=
\left(1+\frac{k}{2}\Xi(t)\right)^{2/k}.
\]
Therefore
\[
\ell_R(t)
=
R^{-1/k}
\left(1+\frac{k}{2}\Xi(t)\right)^{-(k-2)/k},
\]
and
\[
\tau_R(t)
=
R^{-(k-1)/k}
\left(1+\frac{k}{2}\Xi(t)\right)^{(k-2)/k}.
\]
On an interval of bounded length in the $\Xi$ variable, both functions vary only by a factor depending on $k$. Using \eqref{eq:Xi-cap-length} and $\mathrm dt=\ell_R(t)\,\mathrm d\Xi$, we obtain
\[
|\theta|
=
\int_{\Xi(a)}^{\Xi(b)}\ell_R(t(\xi))\,\mathrm d\xi
\asymp_k
\ell_R(t_\theta).
\]

By Lemma~\ref{lem:normal-param-monomial},
\[
|\mathcal N_k(u)-\mathcal N_k(v)|
\asymp_k
|u^{k-1}-v^{k-1}|.
\]
Let $t_{\mathrm{flat}}=R^{-1/k}$ and write $\theta=[a,b]$. If $a\ge t_{\mathrm{flat}}$, then $|\theta|\asymp_k\ell_R(a)\lesssim_k a$, so all points of $\theta$ are comparable to $t_\theta$. The mean value theorem gives
\[
\operatorname{diam}\mathcal N_k(\theta)
\asymp_k
t_\theta^{k-2}|\theta|
\asymp_k
R^{-1/2}t_\theta^{(k-2)/2}
\asymp_k
\tau_R(t_\theta).
\]
If $a<t_{\mathrm{flat}}$, then $\Xi(a)\le\Xi(t_{\mathrm{flat}})=O_k(1)$ and \eqref{eq:Xi-cap-length} implies $\Xi(b)=O_k(1)$. After rescaling $t=t_{\mathrm{flat}}y$, the explicit formula for $\Xi$ shows that $0\le y_a<1$, $y_b=O_k(1)$, and, since $\Xi(b)-\Xi(a)\ge1$, also $y_b-y_a\gtrsim_k1$. Therefore
\[
b^{k-1}-a^{k-1}\asymp_k t_{\mathrm{flat}}^{k-1},
\]
and hence
\[
\operatorname{diam}\mathcal N_k(\theta)
\asymp_k
t_{\mathrm{flat}}^{k-1}
=
R^{-(k-1)/k}
\asymp_k
\tau_R(t_\theta).
\]
Finally, two consecutive caps form an interval of length at most $4$ in the $\Xi$ coordinate; the preceding explicit formulas then give the stated comparability.
\end{proof}

For each canonical cap $\theta$, set $\tau_\theta:=\operatorname{diam}\mathcal N_k(\theta)$. By Lemma~\ref{lem:cap-gauss-diameter}, $\tau_\theta\asymp_k\tau_R(t_\theta)$ for every representative $t_\theta\in\theta$.

Number the caps in order and fix $n_{\mathrm{nb}}=n_{\mathrm{nb}}(k)$ sufficiently large. Write $\theta\sim\theta'$ if their indices differ by at most $n_{\mathrm{nb}}$. By monotonicity of $\mathcal N_k$ and comparability of consecutive caps, we can choose $n_{\mathrm{nb}}$ so that, if $\theta\not\sim\theta'$,
\[
\operatorname{dist}(\mathcal N_k(\theta),\mathcal N_k(\theta'))
\gtrsim_k
\tau_\theta+\tau_{\theta'}.
\]
The resulting partition is an admissible terminal partition in the sense of Definition~\ref{def:terminal-admissible}. The first three properties follow from \eqref{eq:ell-tau-product}, Lemma~\ref{lem:cap-gauss-diameter}, and the preceding separation. For stability under enlargements, if $B\ge1$ is fixed and $u$ lies at distance $O_B(|\theta|)$ from $\theta$, then
\[
u+R^{-1/k}\asymp_{k,B}t_\theta+R^{-1/k}.
\]
Indeed, if the displacement is at most $(t_\theta+R^{-1/k})/2$, the assertion is immediate; in the complementary case, the formula for $\ell_R$ and $|\theta|\asymp_k\ell_R(t_\theta)$ force $t_\theta\lesssim_{k,B}R^{-1/k}$. The explicit formulas for $\ell_R$ and $\tau_R$ then give bounded variation on the enlargement, which by \eqref{eq:Xi-cap-length} meets only $O_{k,B}(1)$ comparable caps.

\subsection{Adapted energy and the local \texorpdfstring{$L^4$}{L4} estimate}

Riesz energies and their truncations are standard objects in potential theory and geometric measure theory; see, for instance, Mattila~\cite{Mattila1995}. In the monomial model define
\begin{equation}\label{eq:energy-adapted-monomial}
\mathcal E_{R,k}(f)
=
\iint_{[0,1]^2}
\frac{|f(u)|^2|f(v)|^2}
{|\mathcal N_k(u)-\mathcal N_k(v)|+\tau_R((u+v)/2)}
\,\mathrm d\sigma(u)\mathrm d\sigma(v).
\end{equation}

\begin{lemma}[Midpoint comparison]\label{lem:tau-midpoint-max}
For $u,v\in[0,1]$,
\[
\tau_R\!\left(\frac{u+v}{2}\right)
\asymp_k
\max\{\tau_R(u),\tau_R(v)\}.
\]
\end{lemma}

\begin{proof}
Assume $u\le v$. Then
\[
\frac12\bigl(v+R^{-1/k}\bigr)
\le
\frac{u+v}{2}+R^{-1/k}
\le
v+R^{-1/k}.
\]
Raising to the power $(k-2)/2$ and multiplying by $R^{-1/2}$ gives the comparability.
\end{proof}

By \eqref{eq:monomial-intrinsic-scales} and the preceding lemma, the general energy from \eqref{eq:energy-general} satisfies
\[
\mathcal E_R^{\Gamma_k}(f)
\asymp_k
\mathcal E_{R,k}(f).
\]
This identification gives an intrinsic explanation for the midpoint-dependent cutoff in \eqref{eq:energy-adapted-monomial}.

\begin{proposition}[Energy--sum equivalence]\label{prop:energy-sum-monomial}
For the canonical partition above,
\begin{align}\label{eq:energy-sum-monomial}
\mathcal E_{R,k}(f)
\asymp_k{}&
\sum_\theta\frac{\|f_\theta\|_2^4}{\tau_\theta}\\
&+
\sum_{\theta\not\sim\theta'}
\iint_{\theta\times\theta'}
\frac{|f(u)|^2|f(v)|^2}{|\mathcal N_k(u)-\mathcal N_k(v)|}
\,\mathrm d\sigma(u)\mathrm d\sigma(v).
\nonumber
\end{align}
\end{proposition}

\begin{proof}
If $\theta\sim\theta'$, comparability of a bounded number of consecutive caps and Lemma~\ref{lem:tau-midpoint-max} give
\[
\tau_R\!\left(\frac{u+v}{2}\right)
\asymp_k
\tau_\theta
\asymp_k
\tau_{\theta'},
\qquad
|\mathcal N_k(u)-\mathcal N_k(v)|\lesssim_k\tau_\theta
\]
for $u\in\theta$, $v\in\theta'$. The neighboring contribution is therefore comparable to
\[
\sum_{\theta\sim\theta'}
\frac{\|f_\theta\|_2^2\|f_{\theta'}\|_2^2}{\tau_\theta},
\]
which, using bounded degree, comparable neighboring scales, and $2ab\le a^2+b^2$, is comparable to
\[
\sum_\theta\frac{\|f_\theta\|_2^4}{\tau_\theta}.
\]
If $\theta\not\sim\theta'$, angular separation gives
\[
|\mathcal N_k(u)-\mathcal N_k(v)|
\gtrsim_k
\tau_\theta+\tau_{\theta'},
\]
and Lemma~\ref{lem:tau-midpoint-max} shows that the cutoff is absorbed by normal separation. Summing the two regimes yields \eqref{eq:energy-sum-monomial}.
\end{proof}

\begin{theorem}[Local $L^4$ estimate for the monomial curve]\label{thm:ball-direct-monomial}
For real $k\ge3$, $R\ge1$, and every $x_0\in\mathbb R^2$,
\[
\int_{B(x_0,R)}|E_{\gamma_k}f(x)|^4\,\mathrm dx
\lesssim_k
\mathcal E_{R,k}(f).
\]
\end{theorem}

\begin{proof}
By the preceding verifications, $\Gamma_k$ satisfies the hypotheses of Theorem~\ref{thm:direct-general-ball}; hence
\[
\int_{B(x_0,R)}|E_{\gamma_k}f|^4
\lesssim_k
\mathcal E_R^{\Gamma_k}(f).
\]
The comparison between $\mathcal E_R^{\Gamma_k}$ and $\mathcal E_{R,k}$ completes the proof.
\end{proof}

The explicit partition defined through $\Xi$ is retained as an explicit description of the terminal partition and as a realization of the general construction.

\subsection{Radius comparison and weights}

\begin{proposition}[Energy comparison]\label{prop:radius-comparison-monomial}
For $\lambda\ge1$,
\[
\mathcal E_{\lambda R,k}(f)
\lesssim_k
\lambda^{(k-1)/k}\mathcal E_{R,k}(f).
\]
\end{proposition}

\begin{proof}
Set $a=R^{-1/k}$. By \eqref{eq:tau-monomial},
\[
\tau_{\lambda R}(t)
=
\lambda^{-1/2}R^{-1/2}
\bigl(t+\lambda^{-1/k}a\bigr)^{(k-2)/2}.
\]
Since $\lambda\ge1$,
\[
t+\lambda^{-1/k}a
\ge
\lambda^{-1/k}(t+a),
\]
and therefore
\[
\tau_{\lambda R}(t)
\ge
\lambda^{-1/2}\lambda^{-(k-2)/(2k)}\tau_R(t)
=
\lambda^{-(k-1)/k}\tau_R(t).
\]
Moreover,
\[
|\mathcal N_k(u)-\mathcal N_k(v)|
\ge
\lambda^{-(k-1)/k}|\mathcal N_k(u)-\mathcal N_k(v)|.
\]
Thus the denominator of the energy at radius $\lambda R$ is at least $\lambda^{-(k-1)/k}$ times the denominator at radius $R$. Integrating the corresponding kernels gives the conclusion.
\end{proof}

\begin{corollary}[Weighted $L^4$ estimate]\label{cor:weighted-direct-monomial}
If $\iota>\frac{k-1}{k}$, then, for every $x_0\in\mathbb R^2$,
\[
\int_{\mathbb R^2}|E_{\gamma_k}f(x)|^4
\left(1+\frac{|x-x_0|}{R}\right)^{-\iota}\,\mathrm dx
\lesssim_{k,\iota}
\mathcal E_{R,k}(f).
\]
\end{corollary}

\begin{proof}
Write
\[
\mathbb R^2
=
B(x_0,2R)
\cup
\bigcup_{j\ge1}
\bigl(B(x_0,2^{j+1}R)\setminus B(x_0,2^jR)\bigr).
\]
On the annulus with index $j$, the weight is $O_{\iota}(2^{-j\iota})$. Applying Theorem~\ref{thm:ball-direct-monomial} at radius $2^{j+1}R$ and then Proposition~\ref{prop:radius-comparison-monomial}, each term is bounded, up to a constant depending on $k$ and $\iota$, by
\[
2^{-j\iota}2^{j(k-1)/k}\mathcal E_{R,k}(f).
\]
The geometric series converges exactly under the hypothesis $\iota>\frac{k-1}{k}$.
\end{proof}

Proposition~\ref{prop:weighted-threshold-necessary}, applied with $m=k-2$ and using the comparison between the general and monomial energies at fixed radius, shows that the uniform inequality fails for $\iota<(k-1)/k$. The endpoint $\iota=(k-1)/k$ remains open.

\section{Multiscale structure and the Frostman regime}
\label{sec:multiscale-frostman}

The $L^4$ estimate for the monomial model reduces the extension problem to the adapted energy from Section~\ref{sec:monomial}. In this section we exploit the explicit description of the terminal cutoff to analyze that energy more precisely. We first obtain a multiscale representation on the normal arc and distinguish the accumulated pair mass from concentration extracted at a single scale. We then impose an angular Frostman condition and derive the resulting growth estimates in $R$.

\subsection{Multiscale representation of the energy}

\subsubsection{Terminal angular scale on the normal arc}

Let $\Omega_k=\mathcal N_k([0,1])$. Since $\mathcal N_k$ is injective, define
\begin{equation}\label{eq:rho-monomial-normal}
\rho_R(\mathcal N_k(t))=\tau_R(t).
\end{equation}

For a finite positive measure $\nu$ on $\Omega_k$, set
\begin{equation}\label{eq:intrinsic-I}
\mathcal I_R(\nu)=
\iint_{\Omega_k^2}
\frac{\mathrm d\nu(\omega)\mathrm d\nu(\omega')}
{|\omega-\omega'|+\max\{\rho_R(\omega),\rho_R(\omega')\}}.
\end{equation}
If $\nu=\nu_f=(\mathcal N_k)_\#(|f|^2\,\mathrm d\sigma)$, Lemma~\ref{lem:tau-midpoint-max} and \eqref{eq:energy-adapted-monomial} give
\[
\mathcal E_{R,k}(f)
\asymp_k
\mathcal I_R(\nu_f).
\]
Thus, once the terminal angular scale $\rho_R$ has been incorporated, the energy can be expressed entirely on the normal arc.

\subsubsection{Layer-cake identity}

Define
\[
\Delta_R(\omega,\omega')
=|\omega-\omega'|+\max\{\rho_R(\omega),\rho_R(\omega')\}
\]
and
\[
\widetilde P_R(r;\nu)
=(\nu\times\nu)\bigl\{(\omega,\omega'):\Delta_R(\omega,\omega')\le r\bigr\}.
\]
If $\nu(\Omega_k)=0$, all normalized statements below are trivial; from now on assume $\nu(\Omega_k)>0$. For $\nu=\nu_f$, in particular, $\nu_f(\Omega_k)=\|f\|_2^2$.

\begin{proposition}[Exact layer-cake identity]\label{prop:layercake}
For every finite positive measure $\nu$ on $\Omega_k$,
\begin{equation}\label{eq:layercake}
\mathcal I_R(\nu)
=
\int_0^\infty \widetilde P_R(r;\nu)\,\frac{\mathrm dr}{r^2}.
\end{equation}
In particular,
\[
\frac{\mathcal I_R(\nu)}{\nu(\Omega_k)^2}
=
\int_0^\infty
\frac{\widetilde P_R(r;\nu)}{\nu(\Omega_k)^2r}\,\frac{\mathrm dr}{r}.
\]
\end{proposition}

\begin{proof}
For $a>0$, $\frac1a=\int_a^\infty\frac{\mathrm dr}{r^2}$. Apply this identity with $a=\Delta_R(\omega,\omega')$ and use Tonelli's theorem.
\end{proof}

The proposition is abstract: once $\rho_R$ is fixed, it uses no further property of the monomial. The specific geometry reappears in the sublevel sets of $\rho_R$ on the normal arc.

\subsubsection{Pair-mass profile}

Define the sublevel set
\[
\Omega_R(r)=\{\omega\in\Omega_k:\rho_R(\omega)\le r\}
\]
and the pair mass at scale $r$ by
\begin{align*}
P_R(r;\nu)=\iint &\mathbf 1_{\{|\omega-\omega'|\le r\}}
\mathbf 1_{\Omega_R(r)}(\omega)
\mathbf 1_{\Omega_R(r)}(\omega')\\
&\mathrm d\nu(\omega)\mathrm d\nu(\omega').
\end{align*}
The normalized pair-mass profile is
\[
\mathfrak q_R(r;\nu)=\frac{P_R(r;\nu)}{\nu(\Omega_k)^2r}.
\]

\begin{lemma}[Comparison of pair masses]\label{lem:P-tildeP}
For $r>0$,
\[
P_R(r/2;\nu)\le\widetilde P_R(r;\nu)\le P_R(r;\nu).
\]
\end{lemma}

\begin{proof}
If $\Delta_R(\omega,\omega')\le r$, then both values of $\rho_R$ are $\le r$ and $|\omega-\omega'|\le r$, which gives the upper bound. Conversely, if both values of $\rho_R$ are $\le r/2$ and $|\omega-\omega'|\le r/2$, then $\Delta_R(\omega,\omega')\le r$.
\end{proof}

Let $r_j=2^{-j}$ and choose $j_R^\ast$ so that
\[
2^{-j_R^\ast}\asymp R^{-(k-1)/k},
\]
the minimal terminal angular scale. Below this scale $P_R(r;\nu)=0$. Above a fixed scale, the layer-cake integral contributes $O(\nu(\Omega_k)^2)$; this is the macroscopic term that appears as $1$ after normalization. In the intermediate range, \Cref{prop:layercake,lem:P-tildeP} and a dyadic discretization give
\begin{equation}\label{eq:dyadic-energy-profile}
\frac{\mathcal E_{R,k}(f)}{\|f\|_2^4}
\asymp_k
1+\sum_{j=0}^{j_R^\ast}\mathfrak q_R(r_j;\nu_f).
\end{equation}
In particular, apart from the macroscopic term, the normalized energy is equivalent, with constants depending on $k$, to the sum of normalized pair masses over the relevant scales.

\subsubsection{Multiscale distribution of the energy}

Set $\nu=\nu_f$ for the remainder of this subsection.

\begin{corollary}[Distribution of the energy across dyadic scales]\label{cor:multiscale-detection}
There exists a constant $C_k\ge1$, depending only on $k$, such that if
\[
\mathcal E_{R,k}(f)\ge\Lambda\|f\|_2^4,
\]
then
\[
\sum_{j=0}^{j_R^\ast}\mathfrak q_R(r_j;\nu_f)
\gtrsim_k \Lambda-C_k.
\]
In particular, for $\Lambda$ larger than a structural constant,
\[
\sum_{j=0}^{j_R^\ast}\mathfrak q_R(r_j;\nu_f)\gtrsim_k\Lambda.
\]
\end{corollary}

\begin{proof}
This is an immediate consequence of \eqref{eq:dyadic-energy-profile}. The constant $C_k$ corresponds to the macroscopic contribution represented by the term $1$ in that equivalence.
\end{proof}

Thus large normalized energy is equivalent, up to the established constants, to a large sum of normalized pair masses across scales. This equivalence quantifies how the energy is distributed across scales without classifying near-extremizers: the sum may be dominated by a single scale or spread across many scales. The next subsection extracts one scale and one angular ball with quantitative concentration. The quantity $n_{\mathrm{eff}}$ defined below measures how many scales are effective for the particular measure.

\subsubsection{Maximal concentration and selection of one scale}

Define the concentration function directly by
\[
\mathfrak C_R(r;\nu)
=
\sup_{\omega\in\Omega_k}
\frac{\nu(B(\omega,r)\cap\Omega_R(r))}{\nu(\Omega_k)r}.
\]
We also have
\[
P_R(r;\nu)
=
\int_{\Omega_R(r)}
\nu(B(\omega,r)\cap\Omega_R(r))\,\mathrm d\nu(\omega),
\]
and hence
\begin{equation}\label{eq:q-less-concentration}
\mathfrak q_R(r;\nu)
=
\frac1{\nu(\Omega_k)}
\int_{\Omega_R(r)}
\frac{\nu(B(\omega,r)\cap\Omega_R(r))}{\nu(\Omega_k)r}\,\mathrm d\nu(\omega)
\le\mathfrak C_R(r;\nu).
\end{equation}
There is no uniform linear inequality in the reverse direction: a set of small mass may have large maximal density without contributing proportionally to the average correlation.

\begin{lemma}[Quadratic relation between correlation and concentration]\label{lem:quadratic-concentration}
For every $r>0$,
\[
\mathfrak q_R(2r;\nu)
\ge
\frac r2\,\mathfrak C_R(r;\nu)^2.
\]
\end{lemma}

\begin{proof}
Choose $\omega_0$ so that its value is within a fixed factor of $\mathfrak C_R(r;\nu)$. Every pair of points in
\[
B(\omega_0,r)\cap\Omega_R(r)
\]
has angular distance $\le2r$ and both values of $\rho_R$ are $\le r\le2r$. Therefore
\[
P_R(2r;\nu)\ge \nu(B(\omega_0,r)\cap\Omega_R(r))^2.
\]
Dividing by $2r\nu(\Omega_k)^2$ and taking the supremum gives the assertion.
\end{proof}

\begin{corollary}[Selection of one scale and angular concentration]\label{cor:one-scale-concentration}
Let $C_k$ be the constant from Corollary~\ref{cor:multiscale-detection}. There exists a constant $C=C(k)\ge1$ such that if
\[
\mathcal E_{R,k}(f)\ge\Lambda\|f\|_2^4,
\]
then there exist a scale $r$ comparable to one of the dyadic scales between the minimal terminal scale and $1$, and a center $\omega_0\in\Omega_k$, such that
\[
\nu_f(B(\omega_0,Cr))
\gtrsim_k
\frac{\Lambda-C_k}{1+\log R}\,r\|f\|_2^2.
\]
The conclusion is nontrivial when $\Lambda$ dominates the structural constant $C_k$.
\end{corollary}

\begin{proof}
There are $O_k(1+\log R)$ relevant dyadic scales. Corollary~\ref{cor:multiscale-detection} and the pigeonhole principle yield a scale $r_j$ with
\[
\mathfrak q_R(r_j;\nu_f)
\gtrsim_k
\frac{\Lambda-C_k}{1+\log R}.
\]
By \eqref{eq:q-less-concentration}, $\mathfrak C_R(r_j;\nu_f)$ satisfies the same lower bound. Choose a center whose value is within a fixed factor of the supremum and absorb the passage from $r_j$ to a comparable radius into the constant $C$.
\end{proof}

The loss $1+\log R$ is the universal loss associated with selecting a single scale. It cannot be removed in general: on a nondegenerate subarc, a measure whose angular density is comparable to length satisfies $P_R(r;\nu)\asymp \nu(\Omega_k)^2r$ at every scale $R^{-1/2}\lesssim r\lesssim1$, so $\mathfrak q_R(r;\nu)\asymp1$ along $\asymp\log R$ dyadic scales.

\subsubsection{Effective number of scales}

For fixed $R$, if the values $\mathfrak q_R(r_j;\nu)$, $0\le j\le j_R^\ast$, are not all zero, define
\[
n_{\mathrm{eff}}(\nu)
=
\frac{\sum_{j=0}^{j_R^\ast}\mathfrak q_R(r_j;\nu)}
{\max_{0\le j\le j_R^\ast}\mathfrak q_R(r_j;\nu)}.
\]
Then
\[
1\le n_{\mathrm{eff}}(\nu)\le j_R^\ast+1\lesssim_k1+\log R.
\]
Since
\[
\max_{0\le j\le j_R^\ast}\mathfrak q_R(r_j;\nu)
=
\frac{\sum_{j=0}^{j_R^\ast}\mathfrak q_R(r_j;\nu)}{n_{\mathrm{eff}}(\nu)},
\]
the same argument as in Corollary~\ref{cor:one-scale-concentration} replaces the universal loss by the effective number of scales of the particular measure: if
\[
\mathcal E_{R,k}(f)\ge\Lambda\|f\|_2^4,
\]
then, whenever the conclusion is nontrivial, there exist $r$ and $\omega_0$ such that, with the same dilation constant $C$ as above,
\[
\nu_f(B(\omega_0,Cr))
\gtrsim_k
\frac{\Lambda-C_k}{n_{\mathrm{eff}}(\nu_f)}\,r\|f\|_2^2.
\]

The two extremes have different interpretations. For a measure approximately uniform on a nondegenerate subarc, $\mathfrak q_R(r_j;\nu)\asymp1$ along $\asymp\log R$ scales, and therefore $n_{\mathrm{eff}}(\nu)\asymp\log R$. By contrast, if the mass is concentrated on a single terminal cap, the sequence $\mathfrak q_R(r_j;\nu)$ is geometrically dominated by the terminal scale and $n_{\mathrm{eff}}(\nu)\asymp1$. Thus $n_{\mathrm{eff}}$ distinguishes energy spread across many scales from energy essentially concentrated at one scale.

\begin{corollary}[Concentration from a large local $L^4$ norm]\label{cor:extension-concentration-monomial}
Let $k\ge3$ be real. There exist constants $\Lambda_0=\Lambda_0(k)\ge1$ and $C=C(k)\ge1$ such that, if $R\ge1$, $x_0\in\mathbb R^2$, $f\not\equiv0$, $\Lambda\ge\Lambda_0$, and
\[
\int_{B(x_0,R)}|E_{\gamma_k}f(x)|^4\,\mathrm dx
\ge
\Lambda\|f\|_2^4,
\]
then there exist a scale $r$ comparable to one of the dyadic scales between the minimal terminal scale and $1$, and a center $\omega_0\in\Omega_k$, such that
\[
\nu_f(B(\omega_0,Cr))
\gtrsim_k
\frac{\Lambda}{1+\log R}\,r\|f\|_2^2.
\]
Moreover, with $n_{\mathrm{eff}}(\nu_f)$ defined above, $1+\log R$ may be replaced by $n_{\mathrm{eff}}(\nu_f)$.
\end{corollary}

\begin{proof}
By Theorem~\ref{thm:ball-direct-monomial}, there exists $A_k\ge1$ such that
\[
\int_{B(x_0,R)}|E_{\gamma_k}f|^4
\le A_k\mathcal E_{R,k}(f).
\]
The hypothesis therefore implies
\[
\mathcal E_{R,k}(f)
\ge A_k^{-1}\Lambda\|f\|_2^4.
\]
Apply Corollary~\ref{cor:one-scale-concentration} with $A_k^{-1}\Lambda$ in place of $\Lambda$. Choosing $\Lambda_0(k)$ so that $A_k^{-1}\Lambda_0$ dominates the macroscopic constant in that corollary gives the first conclusion. The version with $n_{\mathrm{eff}}(\nu_f)$ follows from the argument in the preceding subsection with the same substitution.
\end{proof}

This consequence connects the extension estimate directly to the multiscale information carried by the energy: a large lower bound for the local $L^4$ norm forces concentration of the pushforward measure at some relevant scale. The statement is specific to the monomial model, where the geometry of the sublevel sets $\Omega_R(r)$ is available explicitly.

\subsection{Angular Frostman regime}\label{sec:frostman}

The Frostman analysis is first carried out entirely at the energy level: a Frostman condition on an angular measure is used to estimate the intrinsic energy $\mathcal I_R$. The final corollary then combines that estimate with the monomial $L^4$ estimate to obtain a restriction bound.

Let $\nu$ be a finite measure on $\Omega_k$. If $\nu(\Omega_k)=0$, the statements are trivial; henceforth assume $\nu(\Omega_k)>0$. We use the Frostman condition in its standard sense of upper growth on balls; for the classical framework of Frostman measures and Riesz energies see Mattila~\cite{Mattila1995}. Assume in addition that, for some $F_s>0$,
\begin{equation}\label{eq:frostman-assumption}
\nu(B(\omega,r))\le F_s\,\nu(\Omega_k)\,r^s,
\qquad 0<r\le1,
\qquad 0<s\le1.
\end{equation}
We estimate \eqref{eq:intrinsic-I} from this condition. Since $\Omega_k$ is a fixed arc that can be covered by $O_k(1)$ balls of radius $1$, \eqref{eq:frostman-assumption} implies $F_s\gtrsim_k1$. In particular, contributions with a single power of $F_s$ can be absorbed into a global bound with factor $F_s^2$.

\subsubsection{Geometry of the sublevel sets}

Recall that $\rho_R(\mathcal N_k(t))=\tau_R(t)$. Define the endpoint values of the cutoff by
\[
\rho_0:=\rho_R(\mathcal N_k(0))=R^{-(k-1)/k}
\]
and
\[
\rho_1:=\rho_R(\mathcal N_k(1))
=R^{-1/2}(1+R^{-1/k})^{(k-2)/2}
\asymp_k R^{-1/2}.
\]
The function $t\mapsto\rho_R(\mathcal N_k(t))$ is increasing. Hence, if $r<\rho_0$, the sublevel set $\Omega_R(r)$ is empty; if $r\ge\rho_1$, it equals all of $\Omega_k$; and for $\rho_0\le r\le\rho_1$ there is a unique $t_R(r)\in[0,1]$ such that
\[
\Omega_R(r)=\mathcal N_k([0,t_R(r)]).
\]
The explicit formula for $\tau_R$ gives exactly
\[
\frac r{\rho_0}
=
\bigl(1+t_R(r)R^{1/k}\bigr)^{(k-2)/2},
\]
and therefore
\[
t_R(r)
=
R^{-1/k}
\left[
\left(\frac r{\rho_0}\right)^{2/(k-2)}-1
\right].
\]

Let
\[
L_R(r):=\operatorname{diam}\Omega_R(r)
\]
be the Euclidean diameter of this sublevel set on the normal arc. By the normal--parameter comparison for the monomial model,
\[
L_R(r)\asymp_k t_R(r)^{k-1}.
\]
Consequently,
\begin{equation}\label{eq:active-angular-length}
L_R(r)
\asymp_k
\rho_0
\left[
\left(\frac r{\rho_0}\right)^{2/(k-2)}-1
\right]^{k-1},
\qquad
\rho_0\le r\le\rho_1.
\end{equation}
In particular,
\begin{equation}\label{eq:active-angular-upper}
L_R(r)
\lesssim_k
\rho_0\left(\frac r{\rho_0}\right)^{a_k},
\qquad
a_k=\frac{2(k-1)}{k-2}.
\end{equation}
Formula \eqref{eq:active-angular-length}, not merely its upper bound, is what detects the critical threshold.

\subsubsection{Bound for the pair mass}

By definition of the pair mass,
\[
P_R(r;\nu)
=
\int_{\Omega_R(r)}
\nu\bigl(B(\omega,r)\cap\Omega_R(r)\bigr)
\,\mathrm d\nu(\omega).
\]
The Frostman hypothesis gives, for every center,
\[
\nu(B(\omega,r))\le F_s\,\nu(\Omega_k)\,r^s.
\]
Moreover, since $\Omega_R(r)$ is contained in an angular ball of radius comparable to $L_R(r)$,
\[
\nu(\Omega_R(r))
\le \nu(\Omega_k)\min\{1,F_sL_R(r)^s\}.
\]
Therefore
\begin{equation}\label{eq:frostman-P-bound}
P_R(r;\nu)
\lesssim
F_s\,\nu(\Omega_k)^2r^s\min\{1,F_sL_R(r)^s\}.
\end{equation}
This form retains finer information on the dependence on $F_s$. For the resulting powers of $R$, it suffices in the flat regime to replace the minimum by $F_sL_R(r)^s$.

\subsubsection{The critical threshold}

By Lemma~\ref{lem:P-tildeP} and the layer-cake identity, the contribution of the flat regime is controlled, up to constants, by
\[
\int_{\rho_0}^{\rho_1}
P_R(r;\nu)\,\frac{\mathrm dr}{r^2}.
\]
Using \eqref{eq:frostman-P-bound} and \eqref{eq:active-angular-length}, and setting
\[
x=\frac r{\rho_0},
\qquad
X_R=\frac{\rho_1}{\rho_0}\asymp_k R^{(k-2)/(2k)},
\]
we obtain the bound
\begin{equation}\label{eq:flat-integral}
F_s^2\nu(\Omega_k)^2\rho_0^{2s-1}
\int_1^{X_R}
x^{s-2}
\left(x^{2/(k-2)}-1\right)^{s(k-1)}\,\mathrm dx.
\end{equation}
Near $x=1$, the integrand is integrable. For $x\gg1$, it has order
\[
x^{s(1+a_k)-2}.
\]
The transition between convergence, logarithmic growth, and domination by the upper endpoint occurs when $s(1+a_k)=1$. Thus the critical exponent is
\begin{equation}\label{eq:critical-s}
s_c(k)=\frac1{1+a_k}=\frac{k-2}{3k-4}.
\end{equation}

\begin{theorem}[Frostman energy bounds]\label{thm:frostman-phase}
Let $k\ge3$ and fix $0<s\le1$. Under \eqref{eq:frostman-assumption},
\[
\frac{\mathcal I_R(\nu)}{\nu(\Omega_k)^2}
\lesssim_{k,s}
F_s^2
\begin{cases}
R^{\frac{k-1}{k}(1-2s)},&0<s<s_c(k),\\[1mm]
R^{\frac{k-1}{3k-4}}(1+\log R),&s=s_c(k),\\[1mm]
R^{(1-s)/2},&s_c(k)<s<1,\\[1mm]
1+\log R,&s=1.
\end{cases}
\]
For $s$ near $1$, a uniform form of the nondegenerate contribution, including the macroscopic part, is $1+\frac{R^{(1-s)/2}-1}{1-s}$, with limit $1+\frac12\log R$ as $s\to1$.
\end{theorem}

\begin{proof}
Split the layer-cake representation into three ranges. For $r\gtrsim1$, the contribution is $O(\nu(\Omega_k)^2)$. In the nondegenerate regime $\rho_1\le r\lesssim1$, the sublevel set is all of $\Omega_k$ and
\[
P_R(r;\nu)\lesssim F_s\,\nu(\Omega_k)^2r^s.
\]
Since $F_s\gtrsim_k1$, this contribution is absorbed by a bound with factor $F_s^2$ and is of order
\[
\nu(\Omega_k)^2R^{(1-s)/2}
\]
if $s<1$, while for $s=1$ it is $\nu(\Omega_k)^2(1+\log R)$ after adding the macroscopic contribution.

In the flat regime $\rho_0\le r\le\rho_1$, the bound is \eqref{eq:flat-integral}. If $s<s_c(k)$, the integral in $x$ is uniformly bounded and
\[
\rho_0^{2s-1}
=R^{\frac{k-1}{k}(1-2s)}.
\]
In this range of $s$, this power dominates the nondegenerate contribution. If $s=s_c(k)$, the integral grows like $\log X_R\asymp_k\log R$; moreover,
\[
\frac{k-1}{k}(1-2s_c(k))
=
\frac{k-1}{3k-4}
=
\frac{1-s_c(k)}{2},
\]
so the two regimes have the same power of $R$. If $s_c(k)<s<1$, the integral is dominated by $X_R$ and simplifying the exponents again yields $R^{(1-s)/2}$. If $s=1$, the same flat bound satisfies
\[
F_1^2\nu(\Omega_k)^2\rho_0
\int_1^{X_R}
x^{-1}
\left(x^{2/(k-2)}-1\right)^{k-1}\,\mathrm dx
\lesssim_k
F_1^2\nu(\Omega_k)^2\rho_0X_R^{a_k}
\lesssim_k
F_1^2\nu(\Omega_k)^2,
\]
because $a_k=2(k-1)/(k-2)$, $\rho_0=R^{-(k-1)/k}$, and $X_R\asymp_k R^{(k-2)/(2k)}$. Thus, at the endpoint $s=1$, the flat contribution is bounded and is absorbed by the nondegenerate term of order $F_1^2\nu(\Omega_k)^2(1+\log R)$.

Finally,
\[
\int_{R^{-1/2}}^1r^{s-2}\,\mathrm dr
=
\frac{R^{(1-s)/2}-1}{1-s},
\]
and adding the macroscopic contribution gives the uniform formulation near $s=1$.
\end{proof}

\begin{remark}[Consistency checks for the Frostman bounds]
The preceding formulas admit four useful limiting checks. For $k=2$, outside the degenerate regime of the theorem, the terminal angular scale is constant $R^{-1/2}$ and formally $s_c(2)=0$, in agreement with the absence of a flat regime. As $s\downarrow0$,
\[
\frac{k-1}{k}(1-2s)\longrightarrow\frac{k-1}{k},
\]
which agrees with the maximal growth imposed by the minimal cutoff $\rho_0^{-1}=R^{(k-1)/k}$. At the threshold $s=s_c(k)$,
\[
\frac{k-1}{k}(1-2s_c(k))
=
\frac{k-1}{3k-4}
=
\frac{1-s_c(k)}{2},
\]
and the two powers match. Finally,
\[
\lim_{s\to1^-}
\frac{R^{(1-s)/2}-1}{1-s}
=
\frac12\log R,
\]
so, after adding the macroscopic contribution of order $1$, the uniform formulation near $s=1$ is compatible with growth $1+\log R$.
\end{remark}

\subsubsection{Sharpness of the energy bound at the spatial scale}

Before constructing the extremal examples, we record how angular densities can be realized by data on the curve. Let $\mathrm d\omega$ denote arc-length measure on $\Omega_k$.
If $h\ge0$ belongs to $L^1(\Omega_k,\mathrm d\omega)$, define $f$ by
\[
|f(t)|^2=h(\mathcal N_k(t))\,\kappa_k(t),
\]
where $\kappa_k$ is the curvature of $\gamma_k$. Since the Gauss map is injective and
\[
\mathrm d\omega=\kappa_k(t)\,\mathrm d\sigma(t),
\]
change of variables gives
\[
\nu_f=h\,\mathrm d\omega,
\qquad
\|f\|_2^2=\int_{\Omega_k}h(\omega)\,\mathrm d\omega.
\]
We use $s$-Ahlfors regular in the standard sense: on the relevant structural scales, the mass of a ball centered on the support is comparable to $r^s$; see Mattila~\cite{Mattila1995}. The singular Ahlfors models used below are first regarded as abstract angular measures. To realize them by data on the curve, fix a positive kernel $K_\varepsilon(\omega,\eta)$ adapted to the arc, normalized by $\int K_\varepsilon(\omega,\eta)\,\mathrm d\omega=1$, supported when $|\omega-\eta|\lesssim\varepsilon$, and bounded by $K_\varepsilon\lesssim\varepsilon^{-1}$. Define
\[
h_\varepsilon(\omega)
:=
\int K_\varepsilon(\omega,\eta)\,\mathrm d\nu(\eta),
\qquad
\mathrm d\nu_\varepsilon(\omega)=h_\varepsilon(\omega)\,\mathrm d\omega,
\qquad \varepsilon\ll\rho_0.
\]
This preserves total mass, including for measures whose support reaches an endpoint of the arc. If $\nu$ satisfies an upper $s$-dimensional Frostman bound with fixed Frostman constant, then also
\[
\nu_\varepsilon(B(\omega,r))\lesssim_s \nu(\Omega_k)r^s,
\qquad 0<r\le1,
\]
with a constant uniform in $\varepsilon$. Indeed, for $r\ge\varepsilon$ the convolution is supported in a ball of radius $O(r)$, while for $r<\varepsilon$ the smoothed density satisfies $h_\varepsilon\lesssim_s \nu(\Omega_k)\varepsilon^{s-1}$ and hence
\[
\nu_\varepsilon(B(\omega,r))
\lesssim_s \nu(\Omega_k)\varepsilon^{s-1}r
\le \nu(\Omega_k) r^s.
\]
For $0<s<1$, this regularization does not preserve a uniform lower Ahlfors bound down to scale zero: below $\varepsilon$, the mass of a ball is at most linear in its radius. It does preserve, with structural constants, the Ahlfors lower bounds at scales $r\gtrsim\varepsilon$: if the approximation to the identity moves mass by angular distance at most $\varepsilon$, then for $r\ge2\varepsilon$ all mass originally contained in $B(\omega,r-\varepsilon)$ remains inside $B(\omega,r)$. Since all scales used in the lower bounds of the next proposition satisfy $r\gtrsim\rho_0$ and we choose $\varepsilon\ll\rho_0$, the smoothed measures produce the same orders in $R$ for $\mathcal I_R$. Finally, the smoothed angular density is realized as $\nu_f$ by the preceding construction.

\begin{proposition}[Sharpness of the energy powers]\label{prop:frostman-sharpness}
The powers of $R$ in the energy bound of Theorem~\ref{thm:frostman-phase}, and the logarithmic growth at $s=s_c(k)$, are asymptotically sharp as $R\to\infty$ within the class of $s$-Ahlfors regular angular measures with fixed Ahlfors-regularity constants. This sharpness concerns the energy $\mathcal I_R(\nu)$; Corollary~\ref{cor:frostman-restriction} provides the corresponding upper bound for extension.
\end{proposition}

\begin{proof}[Model constructions]
Consider the four regimes separately.

If $s<s_c(k)$, take an $s$-Ahlfors regular measure supported on an initial subarc whose endpoint is the normal of the flat point. At a scale $r=x\rho_0$ with fixed $x>1$, the sublevel set has angular diameter comparable to $\rho_0$. Its mass is then $\asymp \nu(\Omega_k)\rho_0^s$, and the dyadic profile produces a contribution of order
\[
\nu(\Omega_k)^2\rho_0^{2s-1}
=
\nu(\Omega_k)^2R^{\frac{k-1}{k}(1-2s)}.
\]

In the critical case, again take an $s$-Ahlfors regular measure, with uniform upper and lower bounds, supported on a fixed initial subarc whose endpoint is the normal of the flat point. For constants $C\gg1$ and sufficiently small $c>0$, and
\[
C\rho_0\le r\le c\rho_1,
\]
one has $L_R(r)\gtrsim r$. Let $\omega_{\mathrm{flat}}$ be the normal of the flat point. Restrict the center in the first integration to
\[
B(\omega_{\mathrm{flat}},L_R(r)/4)\cap\operatorname{supp}\nu.
\]
After increasing $C$, every angular ball of radius $r/4$ centered in this set remains inside $\Omega_R(r)$. The lower Ahlfors regularity gives mass $\gtrsim \nu(\Omega_k)L_R(r)^s$ to the set of centers and mass $\gtrsim \nu(\Omega_k)r^s$ to each of these balls. Therefore
\[
P_R(r;\nu)
\gtrsim_{k,s}
\nu(\Omega_k)^2r^sL_R(r)^s.
\]
In the preceding range, \eqref{eq:active-angular-length} gives
\[
L_R(r)
\asymp_k
\rho_0\left(\frac r{\rho_0}\right)^{a_k},
\]
up to structural constants after increasing $C$. Hence
\[
\mathfrak q_R(r;\nu)
\gtrsim_{k,s}
\rho_0^{s(1-a_k)}r^{s(1+a_k)-1}.
\]
When $s=s_c(k)$, one has $s(1+a_k)=1$ and $s(1-a_k)=2s-1$, so each dyadic scale in this range contributes
\[
\mathfrak q_R(r;\nu)
\gtrsim_{k,s}
\rho_0^{2s-1}
=
R^{\frac{k-1}{3k-4}}.
\]
There are $\asymp_k\log R$ dyadic scales between $C\rho_0$ and $c\rho_1$, and \eqref{eq:dyadic-energy-profile} produces the logarithmic factor.

If $s_c(k)<s<1$, place an $s$-Ahlfors regular measure on a nondegenerate subarc, where the cutoff is $\asymp R^{-1/2}$. The truncated Riesz energy has order
\[
\nu(\Omega_k)^2R^{(1-s)/2}.
\]
For $s=1$, an angular density comparable to length measure produces energy of order $\nu(\Omega_k)^2(1+\log R)$ uniformly for $R\ge1$, and hence of order $\nu(\Omega_k)^2\log R$ as $R\to\infty$.

The preceding lower bounds are first obtained with the abstract Ahlfors models. After smoothing at scale $\varepsilon\ll\rho_0$, the lower bounds used in the proof remain valid at every relevant scale $r\gtrsim\rho_0$, although for $s<1$ there is no longer a uniform lower Ahlfors bound when $r\ll\varepsilon$. The upper Frostman bound remains uniform and the smoothed measure is of the form $\nu_f$. Therefore the same orders in $R$ are attained within the class of measures realizable by data on the curve, proving the asserted energy sharpness.
\end{proof}

\begin{remark}\label{rem:Fs-dependence-open}
Estimate \eqref{eq:frostman-P-bound} retains a finer dependence than the final form of the theorem. The uniform bound obtained in the different regimes carries the factor $F_s^2$; determining the optimal dependence on $F_s$ remains open.
\end{remark}

\subsubsection{Consequence for restriction}

\begin{corollary}[Restriction under a Frostman condition]\label{cor:frostman-restriction}
Let $\iota>(k-1)/k$. If $\nu_f$ satisfies \eqref{eq:frostman-assumption}, then, for every $x_0\in\mathbb R^2$,
\[
\int_{\mathbb R^2}|E_{\gamma_k}f(x)|^4
\left(1+\frac{|x-x_0|}{R}\right)^{-\iota}\,\mathrm dx
\lesssim_{k,s,\iota}
F_s^2\|f\|_2^4\,\Phi_{k,s}(R),
\]
where
\[
\Phi_{k,s}(R)=
\begin{cases}
R^{\frac{k-1}{k}(1-2s)},&0<s<s_c(k),\\
R^{\frac{k-1}{3k-4}}(1+\log R),&s=s_c(k),\\
R^{(1-s)/2},&s_c(k)<s<1,\\
1+\log R,&s=1.
\end{cases}
\]
\end{corollary}

\begin{proof}
Combine Corollary~\ref{cor:weighted-direct-monomial}, the equivalence between $\mathcal E_{R,k}(f)$ and $\mathcal I_R(\nu_f)$, and Theorem~\ref{thm:frostman-phase}.
\end{proof}

\printbibliography

\end{document}